\documentclass{amsart}
\usepackage{amssymb}
\usepackage{amsmath}
\newtheorem{theorem}{Theorem}[section]
\newtheorem{lemma}[theorem]{Lemma}
\newtheorem{proposition}[theorem]{Proposition}
\newtheorem{corollary}[theorem]{Corollary}

\newcommand{\Mat}{{\mathbf{M}}}
\newcommand{\supp}{\mathrm{supp}\,}

\newcommand{\vvv}{|\!|\!|}
\newcommand{\Lvvv}{\left|\!\left|\!\left|}
\newcommand{\Rvvv}{\right|\!\right|\!\right|}

\title[Mather sets and joint spectral radii]{Mather sets for sequences of matrices and applications to the study of joint spectral radii}
\author{Ian D. Morris}
\begin{document}
\maketitle
\begin{abstract}
The joint spectral radius of a compact set of $d \times d$ matrices is defined to be the maximum possible exponential growth rate of products of matrices drawn from that set. In this article we investigate the ergodic-theoretic structure of those sequences of matrices drawn from a given set whose products grow at the maximum possible rate. This leads to a notion of Mather set for matrix sequences which is analogous to the Mather set in Lagrangian dynamics. We prove a structure theorem establishing the general properties of these Mather sets and describing the extent to which they characterise matrix sequences of maximum growth. We give applications of this theorem to the study of joint spectral radii and to the stability theory of discrete linear inclusions.

These results rest on some general theorems on the structure of orbits of maximum growth for subadditive observations of dynamical systems, including an extension of the semi-uniform subadditive ergodic theorem of Schreiber, Sturman and Stark, and an extension of a noted lemma of Y. Peres. These theorems are presented in the appendix.
\end{abstract}

\section{Introduction}

Let $A$ be a $d \times d$ complex matrix, and let $\|\cdot\|$ be a norm on $\mathbb{C}^d$. The \emph{spectral radius} of $A$, denoted $\rho(A)$, is defined to be the maximum of the moduli of the eigenvalues of $A$, and satisfies the well-known formula
\[\rho(A)=\lim_{n\to\infty}\|A^n\|^{\frac{1}{n}} = \inf_{n \geq 1} \|A^n\|^{\frac{1}{n}}\]
due to I. Gelfand. By analogy with this formula, given a bounded set $\mathsf{A}$ of complex $d \times d$ matrices we define the \emph{joint spectral radius} of $\mathsf{A}$  to be the quantity
\[\varrho(A):=\lim_{n\to\infty}\sup\left\{\|A_{i_1}A_{i_2}\cdots A_{i_n} \|^{\frac{1}{n}} \colon A_i \in \mathsf{A}\right\}=\inf_{n \geq 1}\sup\left\{\|A_{i_1}A_{i_2}\cdots A_{i_n} \|^{\frac{1}{n}} \colon A_i \in \mathsf{A}\right\}.\]
The existence of this limit and its identity with the above infimum is a straightforward consequence of Fekete's subadditivity lemma, which we include in the appendix as Lemma \ref{Fuckite}. The joint spectral radius was introduced by G.-C. Rota and G. Strang in 1960 (\cite{RS}, later reprinted in \cite{Rotacoll}), and in the last two decades has attracted substantial research attention. This has dealt with its applications, which include control theory \cite{Ba,DHM,Gurvits}, wavelet regularity \cite{DL0}, numerical solutions to differential equations \cite{GZ}, combinatorics \cite{BCJ}, and coding theory \cite{MO}; with algorithms for its computation and estimation \cite{BN,GWZ,Koz2,Madv,PJB}, and questions of its theoretical computability \cite{BT1}; and with its intrinsic properties as a mathematical function \cite{Bochi,HS,Koz4,Wirth1,Wirth2}. 

If $\mathsf{A}$ is a compact set of $d \times d$ complex matrices, then it is not difficult to see that $\varrho(\mathsf{A})$ admits the alternative formulation
\begin{equation}\label{frkrs}\varrho(\mathsf{A})=\sup_{(A_{i_j})_{j =1}^\infty \in \mathsf{A}^{\mathbb{N}}} \limsup_{n \to \infty} \|A_{i_n}\cdots A_{i_1}\|^{\frac{1}{n}}.\end{equation}
That is, $\varrho(\mathsf{A})$ is the fastest possible exponential growth rate of the sequence of partial products of an infinite sequence of elements of $\mathsf{A}$. In this article we are concerned with the problem of understanding, for a given set of matrices, which sequences achieve this fastest possible exponential growth rate. In some form this problem goes back to early research by J. Lagarias and Y. Wang \cite{LW}, and by L. Gurvits \cite{Gurvits}, in which it was asked whether this maximal growth rate is always attained by a periodic sequence.
The general problem of understanding the structure of these fastest-growing sequences of partial products has arisen not only in attempts to answer Lagarias-Wang and Gurvits' question - such as in the articles \cite{BM,HMST,Koz3} - but also in control theory \cite{DHM} and when studying the approximability of the joint spectral radius by spectral radii of finite matrix products \cite{Madv}.

In this article we investigate these orbits of maximal growth from an ergodic-theoretic perspective. Namely, we study those shift-invariant measures on spaces of sequences of matrices for which the pointwise asymptotic growth rate of the sequence of partial products is maximised. We then relate this to some more concrete notions of orbits of maximal growth. Our first theorem, Theorem \ref{simple}, gives some equivalent formulations of the definition of the joint spectral radius in terms of invariant measures. Motivated by earlier research in the fields of Lagrangian dynamics and ergodic optimisation, we go on to consider an object analogous to the \emph{Mather set} in Lagrangian dynamics, which consists of the union of the supports of these growth-maximising invariant measures. Our main result, Theorem \ref{struc}, is a structural theorem which describes the properties of these Mather sets and their relationship with orbits of maximal growth. To motivate this study we give an application of Theorem \ref{struc} to the stability theory of discrete linear inclusions introduced by L. Gurvits \cite{Gurvits}, and use Theorem \ref{struc} to prove some new results relating to the theory of joint spectral radii.

\section{Notation and statement of main theorems}

Before describing our results in more detail, let us establish some notation and conventions. Throughout this document the expression $\log 0$ will be interpreted as being equal to $-\infty$, so that $\log$ extends to a continuous function from $\mathbb{R}_{\geq 0}$ to $\mathbb{R}\cup\{-\infty\}$. We shall use the symbol $\Mat_d(\mathbb{C})$ to denote the set of all $d \times d$ matrices with complex entries. In the remainder of the article the symbol $\|\cdot\|$ will be used to denote both the Euclidean norm on the finite-dimensional space $\mathbb{C}^d$, and also the norm on $\Mat_d(\mathbb{C})$ induced by the Euclidean norm. More general norms on $\mathbb{C}^d$ or $\Mat_d(\mathbb{C})$ shall usually be denoted by $\vvv\cdot\vvv$. We shall call a norm $\vvv \cdot\vvv$ on $\Mat_d(\mathbb{C})$ \emph{submultiplicative} if $\vvv AB\vvv \leq \vvv A \vvv .\vvv B \vvv$ for all $A,B \in \Mat_d(\mathbb{C})$; this property holds in particular for every norm on $\Mat_d(\mathbb{C})$ which is induced by a norm on $\mathbb{C}^d$. We shall say that a norm $\vvv\cdot\vvv$ on $\mathbb{C}^d$ is an \emph{extremal norm} for $\mathsf{A}$ if the induced norm on $\Mat_d(\mathbb{C})$ satisfies $\vvv A \vvv \leq \varrho(\mathsf{A})$ for all $A \in \mathsf{A}$. A set of matrices $\mathsf{A} \subset \Mat_d(\mathbb{C})$ shall be called \emph{relatively product bounded} if there exists a constant $C>1$ such that $\|A_{i_1}\cdots A_{i_n}\| \leq C \varrho(\mathsf{A})^n$ for all $n \geq 1$ and $A_{i_1},\ldots,A_{i_n}\in \mathsf{A}$. It is not difficult to show that a set $\mathsf{A} \subset \Mat_d(\mathbb{C})$ is relatively product bounded if and only if it admits an extremal norm \cite{Jungers,RS}. 

In order to study the set of sequences of elements of $\mathsf{A}$ whose partial products grow at a given rate, it would be most natural to work directly with the sequence space $\mathsf{A}^{\mathbb{N}}$. However, for various reasons this approach is inconvenient. For example, in proving several of our  main results we will find it useful, when considering a sequence of matrices $(A_{i_j})$, to pass to a corresponding sequence of lower-dimensional matrices $(B_{i_j})$ derived from the sequence $(A_{i_j})$, whilst retaining an awareness of the structure of the original sequence. This is problematic since in general the two kinds of sequence will not be related by a bijective correspondence. Furthermore, in other applications we shall be interested in how certain structures change when the set $\mathsf{A}$ itself is perturbed. For these reasons, rather than working directly with $\mathsf{A}^{\mathbb{N}}$ we shall instead choose to work on abstract spaces of \emph{index sequences} which are detached from any a priori relationships with sets of matrices.

Given a compact set of matrices $\mathsf{A}$, we shall say that $\mathcal{I}$ is an \emph{index set} for $\mathsf{A}$, or simply that $\mathsf{A}$ is indexed in $\mathcal{I}$, to mean that $\mathcal{I}$ is a compact metric space equipped with a continuous surjection $\mathcal{I} \to \mathsf{A}$ taking the symbol $i\in\mathcal{I}$ to the matrix $A_i \in \mathsf{A}$. In general this map will not be assumed to be injective. Given an index set $\mathcal{I}$, we shall use the symbol $\Sigma_{\mathcal{I}}$ to denote the metrisable topological space $\mathcal{I}^{\mathbb{N}}$ of all one-sided infinite sequences in $\mathcal{I}$, which we equip with the infinite product topology inherited from $\mathcal{I}$. We define the \emph{shift transformation} $\sigma \colon \Sigma_{\mathcal{I}} \to \Sigma_{\mathcal{I}}$ by $\sigma[(x_i)_{i=1}^\infty]:=(x_{i+1})_{i=1}^\infty$, which is a continuous surjection from $\Sigma_{\mathcal{I}}$ to itself.  When $\mathcal{I}$ equals to $\{1,\ldots,\ell\}$ for some integer $\ell$ we shall simply write $\Sigma_\ell$ for the set $\Sigma_{\mathcal{I}}$. In general we shall leave the structure of $\Sigma_{\mathcal{I}}$ as a metrisable space implicit, but in the special case of finite index sets we shall sometimes find it convenient to equip each $\Sigma_\ell$ with an explicit metric. In these cases we shall consider the metric on each $\Sigma_\ell$ given by
\[d\left[(x_i),(y_i)\right]:=2^{-\inf\{n \geq 1 \colon x_n \neq y_n\}},\]
where $2^{-\infty}$ is interpreted as $0$.

 When a particular shift space $\Sigma_{\mathcal{I}}$ is understood, we shall use the symbol $\mathcal{M}$ to denote the set of all Borel probability measures on $\Sigma_{\mathcal{I}}$, and the symbol $\mathcal{M}_\sigma$ to denote the set of all such measures which are invariant under $\sigma$. We denote the set of all ergodic $\sigma$-invariant measures on $\Sigma_{\mathcal{I}}$ by $\mathcal{E}_\sigma$. We equip $\mathcal{M}$ with the weak-* topology, which is the smallest topology on $\mathcal{M}$ such that the map $\mu \mapsto \int g\,d\mu$ is continuous for every continuous $g \colon \Sigma_{\mathcal{I}} \to \mathbb{R}$. With respect to the weak-* topology $\mathcal{M}$ is compact and metrisable, and $\mathcal{M}_\sigma$ is a closed subset of $\mathcal{M}$ (see e.g. \cite{W}).

We use the following notation to describe the partial products of sequences of matrices associated to elements of $\Sigma_{\mathcal{I}}$. If $\mathsf{A}$ is a compact set of matrices indexed by $\mathcal{I}$, then for each $x=(x_i)_{i=1}^\infty \in \Sigma_{\mathcal{I}}$ and each $n \geq 1$ we define
\[\mathcal{L}_{\mathsf{A}}(x,n):=A_{x_n}A_{x_{n-1}}\cdots A_{x_1}.\]
The function $\mathcal{L}_{\mathsf{A}} \colon \Sigma_{\mathcal{I}} \times \mathbb{N} \to \Mat_d(\mathbb{C})$ is continuous and satisfies the following \emph{cocycle relation}: for all $x \in \Sigma_{\mathcal{I}}$ and $n, m \geq 1$,
\[\mathcal{L}_{\mathsf{A}}(x,n+m) = \mathcal{L}_{\mathsf{A}}(\sigma^nx,m)\mathcal{L}_{\mathsf{A}}(x,n).\]
Since by construction we have $\left\{\mathcal{L}_{\mathsf{A}}(x,n) \colon x \in \Sigma_{\mathcal{I}}\right\} = \left\{A_{i_n}\cdots A_{i_1}\colon A_i \in \mathsf{A}\right\}$ for each $n \geq 1$, the joint spectral radius $\varrho(\mathsf{A})$ is described in terms of $\mathcal{L}_{\mathsf{A}}$ by the expression
\begin{equation}\label{fart}\varrho(\mathsf{A}) = \lim_{n \to \infty}\sup \left\{\|\mathcal{L}_{\mathsf{A}}(x,n)\|^{\frac{1}{n}}\colon x \in \Sigma_{\mathcal{I}}\right\},\end{equation}
and in view of \eqref{frkrs} we further have
\begin{equation}\label{fartier}\varrho(\mathsf{A})= \sup_{x \in \Sigma_{\mathcal{I}}} \limsup_{n \to \infty} \left\|\mathcal{L}_{\mathsf{A}}(x,n)\right\|^{\frac{1}{n}}.\end{equation}
Given a set $\mathsf{A}$ indexed in $\mathcal{I}$ and an ergodic measure $\mu$ on $\Sigma_{\mathcal{I}}$, the subadditive ergodic theorem (given in the appendix as Theorem \ref{SAET}) implies that for $\mu$-a.e. $x \in \Sigma_{\mathcal{I}}$ one has
\[\lim_{n \to \infty} \frac{1}{n} \log \left\|\mathcal{L}_{\mathsf{A}}(x,n)\right\| = \inf_{n \geq 1}\frac{1}{n}\int \log\left\|\mathcal{L}_{\mathsf{A}}(z,n)\right\|\,d\mu(z).\] 
In view of this and \eqref{fartier} it is natural to ask whether, given such a set $\mathsf{A}$ and index set $\mathcal{I}$, there always exists an ergodic measure on $\Sigma_{\mathcal{I}}$ such that $\inf_{n \geq 1}\frac{1}{n}\int \log\|\mathcal{L}_{\mathsf{A}}(x,n)\|\,d\mu=\log\varrho(\mathsf{A})$. The answer to this question is given by the following fundamental result, which provides us with a range of dynamical descriptions of the joint spectral radius $\varrho(\mathsf{A})$:
\begin{theorem}\label{simple}
Let $\mathsf{A} \subset \Mat_d(\mathbb{C})$ be a compact set indexed in $\mathcal{I}$, and let $\vvv\cdot\vvv$ be a submultiplicative norm on $\Mat_d(\mathbb{C})$. Then
\begin{align*}
\log\varrho(\mathsf{A})&=\sup_{\mu \in \mathcal{M}_\sigma} \inf_{n \geq 1}\frac{1}{n}\int_{\Sigma_{\mathcal{I}}} \log\Lvvv\mathcal{L}_{\mathsf{A}}(x,n)\Rvvv\,d\mu(x) \\
&= \inf_{n \geq 1}\sup_{\mu \in \mathcal{M}_\sigma}\frac{1}{n}\int_{\Sigma_{\mathcal{I}}} \log\Lvvv\mathcal{L}_{\mathsf{A}}(x,n)\Rvvv\,d\mu(x)\\
&= \sup_{x \in \Sigma_{\mathcal{I}}} \inf_{n \geq 1} \frac{1}{n}\log\Lvvv\mathcal{L}_{\mathsf{A}}(x,n)\Rvvv.
\end{align*}
In the first two expressions every supremum over $\mathcal{M}_\sigma$ is attained by some ergodic measure, and in the last expression the supremum is attained for at least one $x \in \Sigma_{\mathcal{I}}$. Furthermore, in each of the first two expressions, the infimum over all $n \geq 1$ may be replaced with a limit as $n \to \infty$ without affecting the validity of the expression.
\end{theorem}
This theorem motivates us to study the following two classes of objects. Firstly, given a compact set $\mathsf{A} \subset \Mat_d(\mathbb{C})$ indexed in $\mathcal{I}$, let us define the set of \emph{maximising measures} of $\mathsf{A}$ to be the set of measures on $\Sigma_{\mathcal{I}}$ given by
\[\mathcal{M}_{\max}(\mathsf{A}):=\left\{\mu \in \mathcal{M}_\sigma \colon \inf_{n \geq 1}\frac{1}{n}\int \log \left\|\mathcal{L}_{\mathsf{A}}(x,n)\right\|\,d\mu(x) = \log \varrho(\mathsf{A})\right\}.\]
Secondly, we shall say that $x \in \Sigma_{\mathcal{I}}$ is \emph{strongly extremal} if there exists $\varepsilon>0$ such that $\|\mathcal{L}_{\mathsf{A}}(x,n)\| \geq \varepsilon \varrho(\mathsf{A})^n$ for every $n \geq 1$, and $x$ is \emph{weakly extremal} if $\lim_{n \to \infty} \|\mathcal{L}_{\mathsf{A}}(x,n)\|^{1/n}=\varrho(\mathsf{A})$. These definitions were previously given in a more specialised context in \cite{HMST}; related definitions have been applied independently by other authors \cite{Koz3,Tthesis}, and some comparison with these definitions is undertaken in \S\ref{appratio}. Since every norm on $\Mat_d(\mathbb{C})$ is equivalent to $\|\cdot\|$, the sets of strongly and weakly extremal elements of $\Sigma_{\mathcal{I}}$ are unaffected if a different norm is used in the definition. 

The elementary properties of $\mathcal{M}_{\max}(\mathsf{A})$ are given by the following:
\begin{proposition}\label{rabbit}
Let $\mathsf{A} \subset \Mat_d(\mathbb{C})$ be a compact set indexed in $\mathcal{I}$, and suppose that $\varrho(\mathsf{A})>0$. Then $\mathcal{M}_{\max}(\mathsf{A})$ is compact, convex and nonempty, and its extreme points are precisely its ergodic elements. If $\mu \in \mathcal{M}_\sigma$ and $\vvv\cdot\vvv$ is any norm on $\Mat_d(\mathbb{C})$, then $\mu \in \mathcal{M}_{\max}(\mathsf{A})$ if and only if $\lim_{n \to \infty}\frac{1}{n}\int \log\Lvvv\mathcal{L}_{\mathsf{A}}(x,n)\Rvvv\,d\mu(x)=\log\varrho(\mathsf{A})$. If in addition $\vvv\cdot\vvv$ is submultiplicative, then $\mu \in \mathcal{M}_{\max}(\mathsf{A})$ if and only if $\inf_{n \geq 1}\frac{1}{n}\int \log\Lvvv\mathcal{L}_{\mathsf{A}}(x,n)\Rvvv\,d\mu(x)=\log\varrho(\mathsf{A})$.
\end{proposition}

\emph{A priori}, if $\mu$ is an ergodic maximising measure for $\mathsf{A}$, then by the subadditive ergodic theorem it follows that almost every point in the support of $\mu$ is weakly extremal. By Theorem \ref{simple}, we also know that strongly extremal orbits for $\mathsf{A}$ exist. Beyond this, it is not automatically obvious that the maximising measures and extremal orbits associated to a set $\mathsf{A} \subset \Mat_d(\mathbb{C})$ should enjoy any significant structural relationships. Our main theorem shows that in fact both the maximising measures and the extremal orbits of a set $\mathsf{A} \subset \Mat_d(\mathbb{C})$ are strongly characterised by a certain compact invariant set $Z_{\mathsf{A}} \subseteq \Sigma_{\mathcal{I}}$, which we call the \emph{Mather set} of $\mathsf{A}$. This result echoes some fundamental results in ergodic optimisation and Lagrangian dynamics, as described below. This theorem is the following:
\begin{theorem}\label{struc}
Let $\mathsf{A}\subset \Mat_d(\mathbb{C})$ be a compact set with positive joint spectral radius, and let $\mathcal{I}$ be an index set for $\mathsf{A}$. Define the \emph{Mather set} $Z_{\mathsf{A}} \subseteq \Sigma_{\mathcal{I}}$ associated to $\mathsf{A}$ by
\[Z_{\mathsf{A}}:= \bigcup_{\mu \in \mathcal{M}_{\max}(\mathsf{A})} \supp \mu.\] Then $Z_{\mathsf{A}}$ has the following properties:
\begin{enumerate}
\item
$Z_{\mathsf{A}}$ is equal to the support of some measure $\hat\mu \in \mathcal{M}_\sigma$. In particular, $Z_{\mathsf{A}}$ is compact and nonempty and satisfies $\sigma Z_{\mathsf{A}}=Z_{\mathsf{A}}$.
\item
For every $\mu \in \mathcal{M}_\sigma$ we have $\mu \in \mathcal{M}_{\max}(\mathsf{A})$ if and only if $\mu(Z_{\mathsf{A}})=1$.
\item
Every $x \in Z_{\mathsf{A}}$ which is recurrent with respect to $\sigma$ satisfies the property $\limsup_{n \to \infty} \varrho(\mathsf{A})^{-n}\rho(\mathcal{L}_{\mathsf{A}}(x,n))=1$.
\item
Every $x \in Z_{\mathsf{A}}$ is strongly extremal. If in addition $\mathsf{A}$ is relatively product bounded, then for every extremal norm $\vvv \cdot \vvv$ for $\mathsf{A}$ we have $\vvv\mathcal{L}_{\mathsf{A}}(x,n)\vvv = \varrho(\mathsf{A})^n$ for all $x \in Z_{\mathsf{A}}$ and $n \geq 1$. 
\item
Let $d$ be a metric which generates the topology of $\Sigma_{\mathcal{I}}$, and define $\mathrm{dist}(x,Z_{\mathsf{A}}):=\inf\{d(x,y) \colon y \in Z_{\mathsf{A}}\}$ for every $x \in \Sigma_{\mathcal{I}}$. If $z \in \Sigma_{\mathcal{I}}$  is weakly extremal, then $\lim_{n \to \infty}(1/n)\sum_{k=0}^{n-1}\mathrm{dist}(\sigma^kz,Z_{\mathsf{A}}) = 0$.
\end{enumerate}
\end{theorem}

\emph{Remark}. The definition of $Z_{\mathsf{A}}$ and parts (i) and (ii) of Theorem \ref{struc} were suggested by well-known results of J. Mather and R. Ma\~n\'e in Lagrangian dynamics \cite{Mane92,Mather91}; similar results have been obtained in other dynamical contexts by T. Bousch \cite{Bousch01} and by G. Contreras \emph{et al.} \cite{CLT}. 
Additionally, for the one-parameter family of pairs of matrices $\mathsf{A}_\alpha=\left\{A_1^{(\alpha)},A_2^{(\alpha)}\right\}$ given by
\[A_1^{(\alpha)}=\left(\begin{array}{cc}1 & 1\\0 & 1\end{array}\right),\qquad A_2^{(\alpha)}=\left(\begin{array}{cc}\alpha &0 \\\alpha & \alpha\end{array}\right),\]
where $\alpha \in [0,1]$, it is shown in \cite[Theorem 2.3]{HMST} that there exists for each $\alpha$ a set $X_{\mathfrak{r}(\alpha)} \subset \Sigma_2$ which satisfies properties (iii)-(v) above. Theorem \ref{struc} above can therefore also be seen as a partial generalisation of \cite[Theorem 2.3]{HMST} to arbitrary compact sets of square matrices.

Theorem \ref{struc} allows us to prove a number of theorems relating to optimal growth of matrix sequences, which are described in detail in subsequent sections of this article.  In \S\ref{appmarkov}, we prove a theorem which relates two different notions of stability for discrete linear inclusions introduced by L. Gurvits in the influential article \cite{Gurvits}. In \S\ref{appratio} we prove a proposition which unifies various definitions of the `1-ratio' of a pair of matrices given respectively by V. S. Kozyakin, by T. Bousch and J. Mairesse, and by K. Hare \emph{et al}. Using this unified definition, we prove a result on the continuity of 1-ratios which generalises results by each of these authors. Connections with the Lagarias-Wang finiteness property are discussed. Finally, in \S\ref{appbara} we prove a theorem dealing with the classification of Barabanov norms, and prove uniqueness of these norms for a family of pairs of matrices studied in \cite{BTV,HMST,Tthesis}. We defer detailed descriptions of these results to the relevant sections.

The proofs of Theorem \ref{simple}, Proposition \ref{rabbit} and Theorem \ref{struc} are given in  \S\ref{pfs}, following some preparatory results in \S\ref{prep}. The first two results in particular are dependent on general results on subadditive sequences of functions, which are presented in the appendix.

\section{Preliminaries to the proof of Theorem \ref{struc}}\label{prep}

In this section we prove some preliminary results needed for the proof of Theorem \ref{struc}, the purpose of which is to allow us to deal with the case in which $\mathsf{A}$ is not relatively product bounded. In this and the following section, the symbol $|L|$ will be used to denote the norm of the matrix $L$ given by the maximum of the moduli of the entries of $L$. This is the only norm which we shall apply to matrices which are strictly rectangular.

We begin with the following key definition. Let $\mathsf{A} \subset \Mat_{d}(\mathbb{C})$ be compact, and let $\mathcal{I}$ be an index set for $\mathsf{A}$. If $\mathsf{A}^{(1)}=\left\{A^{(1)}_i \colon i \in \mathcal{I}\right\}$, $\mathsf{A}^{(2)}=\left\{A^{(2)}_i \colon i \in \mathcal{I}\right\}$ and $\mathsf{B}=\left\{B_i \colon i \in \mathcal{I}\right\}$ are sets of matrices respectively of dimension $k \times k$, $(d-k)\times(d-k)$ and $k \times (d-k)$, where $0<k<d$, we say that $\left(\mathsf{A}^{(1)},\mathsf{A}^{(2)},\mathsf{B},M\right)$ is an \emph{upper triangularisation} of $\mathsf{A}$ if the relation
\[M^{-1}A_iM = \left(\begin{array}{cc}A_i^{(1)}&B_i\\0&A_i^{(2)}\end{array}\right)\]
is satisfied for all $i \in \mathcal{I}$. Note that if this is the case then each of the maps $i \mapsto A_i^{(1)}$, $i \mapsto A^{(2)}$, $i \mapsto B_i$ is continuous. The first of the two main results of this section is the following, which guarantees the existence of upper triangularisations with favourable properties:
\begin{proposition}\label{ERXX}
Suppose that $\mathsf{A} \subset \Mat_d(\mathbb{C})$ has nonzero joint spectral radius and is not relatively product bounded. Then there exists an upper triangularisation $\left(\mathsf{A}^{(1)},\mathsf{A}^{(2)},\mathsf{B},M\right)$ of $\mathsf{A}$ such that $\varrho\left(\mathsf{A}^{(1)}\right)=\varrho\left(\mathsf{A}^{(2)}\right)=\varrho(\mathsf{A})$ and $\mathsf{A}^{(1)}$ is relatively product bounded.
\end{proposition}

In order to prove Proposition \ref{ERXX} we require two lemmas. The following lemma is by now well-known to joint spectral radius researchers; it appears to originate in \cite{Ba}. For a proof we refer the reader to the book of R. Jungers \cite{Jungers}.
\begin{lemma}\label{Elsnerredux}
Suppose that $\mathsf{A}$ is not relatively product bounded. Then there exists a subspace $U \subset \mathbb{C}^d$ such that $0 < \dim U < d$ and $A_iU \subseteq U$ for all $A_i \in \mathsf{A}$.
\end{lemma}
The statement of the following lemma was suggested to the author during the reviewing process for the article \cite{Madv} by an anonymous reviewer. To the best of the author's knowledge it has not previously appeared in print.
\begin{lemma}\label{Els2}
Suppose that $\mathsf{A}$ is not relatively product bounded. Then there exists an upper triangularisation $\left(\mathsf{A}^{(1)},\mathsf{A}^{(2)},\mathsf{B},M\right)$ for $\mathsf{A}$ such that $\varrho(\mathsf{A}^{(1)})=\varrho(\mathsf{A})$.
\end{lemma}
\begin{proof}
By Lemma \ref{Elsnerredux} there exists a subspace $U$ of $\mathbb{C}^d$ such that $A_iU \subseteq U$ for all $i\in \mathcal{I}$, and $0 < \dim U <d$. Let $r$ be the largest possible dimension of such a subspace $U$, and fix an invariant subspace $U$ of that dimension. Let $v_1,\ldots,v_d$ be a basis for $\mathbb{C}^d$ having the property that $v_1,\ldots,v_r$ is a basis for $U$, and let $M$ be the corresponding change-of-basis matrix. 
Clearly for each $i \in \mathcal{I}$ we have
\[M^{-1}A_iM = \left(\begin{array}{cc}A_i^{(1)}&B_i\\0&A_i^{(2)}\end{array}\right)\]
for some $r \times r$ matrix $A_i^{(1)}$, $r\times (d-r)$ matrix $B_i$, and $(d-r)\times (d-r)$ matrix $A_i^{(2)}$. Define $\mathsf{B}=\left\{B_i \colon i \in \mathcal{I}\right\}$ and $\mathsf{A}^{(k)}=\left\{A_i^{(k)} \colon i \in \mathcal{I}\right\}$ for $k=1,2$, so that $\left(\mathsf{A}^{(1)},\mathsf{A}^{(2)},\mathsf{B},M\right)$ is an upper triangularisation of $\mathsf{A}$. 

We claim that there does not exist a subspace $V$ of $\mathbb{C}^{d-r}$ such that $A_i^{(2)} V \subseteq V$ for all $i \in \mathcal{I}$, and $0 < \dim V < d-r$. To see this it suffices to note that if such a subspace $V$ were to exist, then $U \oplus MV$ 
would be an invariant subspace for $\mathsf{A}$ with dimension strictly between $r$ and $d$, which is impossible by the definition of $r$. We conclude via Lemma \ref{Elsnerredux} that $\mathsf{A}^{(2)}$ is relatively product bounded.

We may now prove that $\varrho\left(\mathsf{A}^{(1)}\right)=\varrho(\mathsf{A})$. For a contradiction let us assume that $\varrho\left(\mathsf{A}^{(1)}\right)<\varrho(\mathsf{A})$. Since $\mathsf{A}^{(2)}$ is relatively product bounded with joint spectral radius not greater than $\varrho(\mathsf{A})$, and since $\varrho\left(\mathsf{A}^{(1)}\right)<\varrho(\mathsf{A})$, we may choose constants $C_1,C_2>1$ and $\theta \in (0,\varrho(\mathsf{A}))$ such that for all  $n \geq 1$ we have $\left|A_{i_1}^{(1)}\cdots A_{i_n}^{(1)}\right| \leq C \theta^n$ and $\left|A_{i_1}^{(2)}\cdots A_{i_n}^{(2)}\right| \leq C_2 \varrho(\mathsf{A})^n$ for every $(i_1,\ldots,i_n)\in\mathcal{I}^n$. Using the fact that $\mathsf{B}$ is compact, we may choose a constant $C_3>1$ such that if $L_1 \in \Mat_r(\mathbb{C})$ and $L_2 \in \Mat_{d-r}(\mathbb{C})$ then $|L_1 B_i L_2| \leq C_3 |L_1|.|L_2|$ for all $i \in \mathcal{I}$. As such, for every $n \geq 1$ and $(i_1,\ldots,i_n) \in \mathcal{I}^n$ we have
\[\left|M^{-1}A_{i_1}\cdots A_{i_n}M\right| = \left| \left(\begin{array}{cc}A^{(1)}_{i_1}\cdots A^{(1)}_{i_n}&\sum_{k=1}^n A_{i_1}^{(1)} \cdots A_{i_{k-1}}^{(1)} B_{i_k} A_{i_{k+1}}^{(2)} \cdots A_{i_n}^{(2)}\\ 0&A^{(2)}_{i_1}\cdots A^{(2)}_{i_n}\end{array}\right)\right|\]
\[\leq \max\left\{\left|A^{(1)}_{i_1}\cdots A^{(1)}_{i_n}\right|, \left|A^{(2)}_{i_1}\cdots A^{(2)}_{i_n}\right|, \sum_{k=1}^n \left|A_{i_1}^{(1)} \cdots A_{i_{k-1}}^{(1)} B_{i_k} A_{i_{k+1}}^{(2)} \cdots A_{i_n}^{(2)}\right|\right\}.\]
Since
\begin{align*}\sum_{k=1}^n \left|A_{i_1}^{(1)} \cdots A_{i_{k-1}}^{(1)} B_{i_k} A_{i_{k+1}}^{(2)} \cdots A_{i_n}^{(2)}\right| &\leq \sum_{k=1}^n C_3\left|A_{i_1}^{(1)} \cdots A_{i_{k-1}}^{(1)}\right|.\left|A_{i_{k+1}}^{(2)} \cdots A_{i_n}^{(2)}\right|\\
&\leq \sum_{k=1}^n C_1C_2C_3 \theta^{k-1} \varrho\left(\mathsf{A}\right)^{n-k}\\
&<C_1C_2C_3 \varrho(\mathsf{A})^{n-1}\sum_{k=0}^\infty \left(\frac{\theta}{\varrho(\mathsf{A})}\right)^k\\
&= \varrho(\mathsf{A})^n.\frac{C_1C_2C_3}{\varrho(\mathsf{A})-\theta}\end{align*}
and
\[\max\left\{\left|A^{(1)}_{i_1}\cdots A^{(1)}_{i_n}\right|, \left|A^{(2)}_{i_1}\cdots A^{(2)}_{i_n}\right|\right\}\leq \varrho(\mathsf{A})^n\max\{C_1,C_2\},\]
we conclude that
\[\max\left\{\varrho(\mathsf{A})^{-n}|M^{-1}A_{i_1}\cdots A_{i_n}M|\colon (i_1,\ldots,i_n) \in \mathcal{I}^n\right\} \leq \max\left\{\frac{C_1C_2C_3}{\varrho(\mathsf{A})-\theta},C_1,C_2\right\}\]
for every $n \geq 1$. The invertibility of $M$ implies that the functional $A \mapsto |M^{-1}AM|$ is a norm on $\Mat_d(\mathbb{C})$, and hence we have shown that $\mathsf{A}$ is relatively product bounded. This contradicts the hypotheses of the lemma, and we conclude that necessarily $\varrho\left(\mathsf{A}^{(1)}\right)=\varrho(\mathsf{A})$. The proof is complete.
\end{proof}

\begin{proof}[of Proposition \ref{ERXX}] 
Let $\left(\mathsf{A}^{(1)},\mathsf{A}^{(2)},\mathsf{B},M\right)$ be an upper triangularisation of $\mathsf{A}$ such that $\varrho\left(\mathsf{A}^{(1)}\right)=\varrho(\mathsf{A})$, which exists by Lemma \ref{Els2}, and suppose that the matrices comprising $\mathsf{A}^{(1)}$ have the smallest possible dimension $r$ for which this relationship can hold. It follows in particular that there cannot exist an upper triangularisation $\left(\mathsf{C}^{(1)},\mathsf{C}^{(2)},\mathsf{D},\tilde M\right)$ of $\mathsf{A}^{(1)}$ such that $\varrho\left(\mathsf{A}^{(1)}\right)=\varrho\left(\mathsf{C}^{(1)}\right)$, since this would lead to a new upper triangularisation of $\mathsf{A}$ in which the upper-left matrices have smaller than the minimum possible dimension. By Lemma \ref{Els2} it follows that $\mathsf{A}^{(1)}$ must be relatively product bounded.

Now let us suppose for a contradiction that $\varrho\left(\mathsf{A}^{(2)}\right)<\varrho(\mathsf{A})$. We make estimates similar to those in the proof of Lemma \ref{Els2}. Choose a constant $C_3>1$ such that if $L_1 \in \Mat_r(\mathbb{C})$, $L_2 \in \Mat_{d-r}(\mathbb{C})$ and $i \in \mathcal{I}$ then $|L_1B_iL_2| \leq C_3|L_1|.|L_2|$, and choose constants $C_1,C_2>1$ and $\theta<\varrho(\mathsf{A})$ such that for each $n \geq 1$ and $(i_1,\ldots,i_n)\in\mathcal{I}^n$ we have $\left|A_{i_1}^{(1)}\cdots A_{i_n}^{(1)}\right| \leq C_1 \varrho(\mathsf{A})^n$ and $\left|A_{i_1}^{(2)}\cdots A_{i_n}^{(2)}\right| \leq C_2 \theta^n$. Given any $n \geq 1$ and $(i_1,\ldots,i_n)\in\mathcal{I}^n$, the quantity
\[ \max\left\{\left|A^{(1)}_{i_1}\cdots A^{(1)}_{i_n}\right|, \left|A^{(2)}_{i_1}\cdots A^{(2)}_{i_n}\right|, \sum_{k=1}^n \left|A_{i_1}^{(1)} \cdots A_{i_{k-1}}^{(1)} B_{i_k} A_{i_{k+1}}^{(2)} \cdots A_{i_n}^{(2)}\right|\right\}\]
is as before an upper bound for $\left|M^{-1}A_{i_1}\cdots A_{i_n}M\right|$, and in a similar manner to the previous proof we obtain the estimates
\[\sum_{k=1}^n \left|A_{i_1}^{(1)} \cdots A_{i_{k-1}}^{(1)} B_{i_k} A_{i_{k+1}}^{(2)} \cdots A_{i_n}^{(2)}\right|\leq \sum_{k=1}^nC_1C_2C_3 \varrho(\mathsf{A})^{k-1}\theta^{n-k}<\varrho(\mathsf{A})^n\frac{C_1C_2C_3}{\varrho(\mathsf{A})-\theta}\]
and
\[\max\left\{\left|A^{(1)}_{i_1}\cdots A^{(1)}_{i_n}\right|, \left|A^{(2)}_{i_1}\cdots A^{(2)}_{i_n}\right|\right\}\leq \varrho(\mathsf{A})^n\max\{C_1,C_2\}.\]
It follows as before that the quantity
\[\max\left\{\varrho(\mathsf{A})^{-n}\|A_{i_1}\cdots A_{i_n}\|\colon (i_1,\ldots,i_n) \in \mathcal{I}^n\right\}\]
is bounded independently of $n$, contradicting the hypothesis that $\mathsf{A}$ is not relatively product bounded. We conclude that necessarily $\varrho\left(\mathsf{A}^{(2)}\right)=\varrho(\mathsf{A})$ as claimed. This completes the proof of the proposition.
\end{proof}

The second major result of this section is the following, which allows us to apply the properties of the upper triangularisations given by Proposition \ref{ERXX} to the comparison of growth rates along certain sequences.

\begin{proposition}\label{ergmeasures}
Let $\mathsf{A}$ be a compact set of matrices indexed in $\mathcal{I}$, and suppose that $\left(\mathsf{A}^{(1)}, \mathsf{A}^{(2)}, \mathsf{B}, M\right)$ is an upper triangularisation of $\mathsf{A}$. Then for every $\mu \in \mathcal{E}_\sigma$,
\[\inf_{n \geq 1}\frac{1}{n}\int\log \|\mathcal{L}_{\mathsf{A}}(x,n)\|\,d\mu(x) = \max_{j=1,2} \inf_{n \geq 1}\frac{1}{n}\int\log \|\mathcal{L}_
{\mathsf{A}^{(j)}}(x,n)\|\,d\mu(x).\]
\end{proposition}
In order to prove the proposition we require the following abstract lemma:
\begin{lemma}\label{dfbugsdhfui}
Let $(X,\mathcal{F},\nu)$ be a probability space and let $T \colon X \to X$ be an ergodic measure-preserving transformation. Suppose that $(f_n^1)$, $(f^2_n)$ are sequences of measurable functions from $X$ to $\mathbb{R} \cup \{-\infty\}$ such that $f_{n+m}^j \leq f_n^j \circ T^m + f_m^j$ $\nu$-a.e. for all $n, m \geq 1$, for each $j \in\{1,2\}$, and such that each $f_n^j$ is bounded above by an integrable function. Define $\lambda_j:=\inf_{n \geq 1}\frac{1}{n}\int f_n^j\,d\nu$ and $f_0^j \equiv 0$ for each $j$. Then for $\nu$-a.e. $x \in X$,
\[\liminf_{n \to \infty}\frac{1}{n}\max_{1 \leq k \leq n}\left(f_{n-k}^1(T^kx) + f_{k-1}^2(x)\right) \leq \max\{\lambda_1,\lambda_2\}.\]
\end{lemma}
\begin{proof}
By passing to the natural extension of $(X,T)$ if necessary, we may without loss of generality assume that $T$ is invertible. We will show that for each integer $\ell \geq 1$,
\begin{equation}\label{ello}\nu\left(\left\{x \in X \colon \liminf_{n \to \infty} \frac{1}{n}\max_{1 \leq k \leq n}\left(f^1_{n-k}(T^kx) + f_{k-1}^2(x)\right) \leq \max\{\lambda_1,\lambda_2\}+\frac{1}{\ell}\right\}\right)=1,\end{equation}
which suffices to prove the lemma.

Applying Theorem \ref{SAET} to the sequence $(f_n^2)$ and the transformation $T$ we find that $n^{-1}f_n^2 \to \lambda_2$ $\nu$-almost everywhere. On the other hand, our hypotheses imply that $f^1_{n+m} \circ T^{-n-m} \leq f_n^1 \circ T^{-n-m} + f_m^1 \circ T^{-m}$ almost everywhere for each $n,m \geq 1$, and hence we may apply Theorem \ref{SAET} to the sequence $(f_n^1 \circ T^{-n})$ and the transformation $T^{-1}$ to deduce that $f_n^1 \circ T^{-n} \to \lambda_1$ $\nu$-a.e. Fix $\ell \geq 1$, and for each $x \in X$ define $C_\ell^1(x):=\sup\{f_n^1(T^{-n}x)-n(\lambda_1+\ell^{-1}) \colon n\geq 0\}$ and $C_\ell^2(x):=\sup\{f_n^2(x)-n(\lambda_2+\ell^{-1}) \colon n \geq 0\}$. Note that each $C_\ell^j(x)$ is finite for $\nu$-a.e. $x \in X$. For each $x \in X$ and $n \geq 1$ we then have
\begin{eqnarray}\label{gank}\lefteqn{\max_{1 \leq k \leq n} \left(f_{n-k}^1(T^kx) + f_{k-1}^2(x)\right) }\\
& \leq & \max_{1 \leq k \leq n} \left(C^1_\ell(T^nx) + (n-k)\left(\lambda_2+\ell^{-1}\right) + C^2_\ell(x)+(k-1)\left(\lambda_1+\ell^{-1}\right)\right)\nonumber\\
& = &  C^1_\ell(T^nx)+C^2_\ell(x)+(n-1)\max\{\lambda_1,\lambda_2\}+(n-1)\ell^{-1}.\nonumber\end{eqnarray}
Now choose some $K \in \mathbb{R}$ such that $\nu(\{x \in X \colon C_\ell^1(x) \leq K\})>0$. By ergodicity we have $C_\ell^1(T^nx) \leq K$ infinitely often for $\nu$-a.e. $x$, and it follows that $\liminf_{n \to \infty} n^{-1}C^1_\ell(T^nx) \leq 0$ $\nu$-a.e. Combining this fact with the inequality \eqref{gank} yields \eqref{ello},  and since $\ell \geq 1$ is arbitrary the lemma follows.
\end{proof}
\begin{proof}[ of Proposition \ref{ergmeasures}.] Let us define
\[\lambda:=\inf_{n \geq 1}\frac{1}{n}\int \log\|\mathcal{L}_{\mathsf{A}}(x,n)\|\,d\mu(x)\]
and
\[\lambda_j:=\inf_{n \geq 1}\frac{1}{n}\int \log\|\mathcal{L}_{\mathsf{A}^{(i)}}(x,n)\|\,d\mu(x)\]
for $j=1,2$. From the definition of upper triangularisation, we have for each $n \geq 1$ and all $x \in \Sigma_{\mathcal{I}}$,
\begin{equation}\label{balls}|M^{-1}\mathcal{L}_{\mathsf{A}}(x,n)M| = \left| \left(\begin{array}{cc}\mathcal{L}_{\mathsf{A}^{(1)}}(x,n)&\sum_{k=1}^n \mathcal{L}_{\mathsf{A}^{(1)}}(\sigma^kx,n-k) B_{x_k} \mathcal{L}_{\mathsf{A}^{(2)}}(x,k-1)\\ 0&\mathcal{L}_{\mathsf{A}^{(2)}}(x,n)\end{array}\right)\right|,\end{equation}
so that in particular
\begin{equation}\label{bellend}\frac{1}{n}\log |M^{-1}\mathcal{L}_{\mathsf{A}}(x,n)M| \geq \max_{j \in \{1,2\}}\frac{1}{n}\log|\mathcal{L}_{\mathsf{A}^{(j)}}(x,n)|.\end{equation}
Since $M$ is invertible, the functional $A \mapsto |M^{-1}AM|$ defines a norm on $\Mat_d(\mathbb{C})$, and so we may deduce from \eqref{bellend} the inequality
\begin{align*}\lambda = \inf_{n \geq 1}\frac{1}{n}\int \log\|\mathcal{L}_{\mathsf{A}}(x,n)\|\,d\mu(x) &= \lim_{n \to \infty} \frac{1}{n}\int \log\|\mathcal{L}_{\mathsf{A}}(x,n)\|\,d\mu(x)\\
&= \lim_{n \to \infty} \frac{1}{n}\int \log|M^{-1}\mathcal{L}_{\mathsf{A}}(x,n)M|\,d\mu(x)\\
&\geq \liminf_{n \to \infty} \max_{j \in \{1,2\}}\frac{1}{n}\int \log|\mathcal{L}_{\mathsf{A}^{(j)}}(x,n)|\,d\mu(x)\\
&\geq\max_{j \in \{1,2\}}\liminf_{n \to \infty} \frac{1}{n}\int \log\|\mathcal{L}_{\mathsf{A}^{(j)}}(x,n)\|\,d\mu(x)\\
&=\max_{j \in \{1,2\}}\inf_{n \geq 1} \frac{1}{n}\int \log\|\mathcal{L}_{\mathsf{A}^{(j)}}(x,n)\|\,d\mu(x)\\
&=\max\{\lambda_1,\lambda_2\},\end{align*}
where we have used Lemma \ref{Fuckite} to identify the infima with the corresponding limits. Let us now prove the reverse inequality $\lambda \leq \max\{\lambda_1,\lambda_2\}$. As a consequence of \eqref{balls}, the quantity
\[\max\left\{\left|\mathcal{L}_{\mathsf{A}^{(1)}}(x,n)\right|,  \left|\mathcal{L}_{\mathsf{A}^{(2)}}(x,n)\right|, \sum_{k=1}^n \left| \mathcal{L}_{\mathsf{A}^{(1)}}(\sigma^kx,n-k) B_{x_k} \mathcal{L}_{\mathsf{A}^{(2)}}(x,k-1)\right|\right\}\]
is an upper bound for $|M^{-1}\mathcal{L}_{\mathsf{A}}(x,n)M|$ for every $x$ and $n$. Choose $C>1$ such that $|L_1B_iL_2| \leq C \|L_1\|.\|L_2\|$ for all $i \in \mathcal{I}$, $L_1 \in \Mat_r(\mathbb{C})$ and $L_2 \in \Mat_{d-r}(\mathbb{C})$. Applying Lemma \ref{dfbugsdhfui} with $f^j_n(x):=\log \|\mathcal{L}_{\mathsf{A}^{(j)}}(x,n)\|$ we obtain for $\mu$-a.e. $x \in \Sigma_{\mathcal{I}}$,
\begin{eqnarray*}\lefteqn{\liminf_{n \to \infty} \frac{1}{n}\log\left(\sum_{k=1}^n \left| \mathcal{L}_{\mathsf{A}^{(1)}}(\sigma^kx,n-k) B_{x_k} \mathcal{L}_{\mathsf{A}^{(2)}}(x,k-1)\right|\right)}\\
& \leq & \liminf_{n \to \infty} \frac{1}{n}\log \left(Cn. \max_{1 \leq k\leq n} \left\| \mathcal{L}_{\mathsf{A}^{(1)}}(\sigma^kx,n-k)\right\|.\left\| \mathcal{L}_{\mathsf{A}^{(2)}}(x,k-1)\right\|\right) \leq \max\{\lambda_1,\lambda_2\}.\end{eqnarray*}
Applying Theorem \ref{SAET} to each sequence $(f^j_n)$ and making use of the equivalence of norms on finite-dimensional spaces we obtain for $\mu$-a.e. $x$
\[\lim_{n \to \infty} \frac{1}{n} \max_{j\in\{1,2\}}\log\left|\mathcal{L}_{\mathsf{A}^{(j)}}(x,n)\right| = \lim_{n \to \infty}\frac{1}{n} \max_{j \in \{1,2\}} \log \|\mathcal{L}_{\mathsf{A}^{(j)}}(x,n)\|=\max\{\lambda_1,\lambda_2\}.\]
Combining these estimates, it follows that for $\mu$-a.e. $x \in \Sigma_{\mathcal{I}}$
\[\liminf_{n \to \infty} \frac{1}{n} \log\|\mathcal{L}_{\mathsf{A}}(x,n)\|=\liminf_{n \to \infty}\frac{1}{n} \log\left|M^{-1}\mathcal{L}_{\mathsf{A}}(x,n)M\right| \leq \max\{\lambda_1,\lambda_2\}.\]
Applying Theorem \ref{SAET} once more with $f_n(x):=\log\|\mathcal{L}_{\mathsf{A}}(x,n)\|$  we conclude that for $\mu$-a.e. $x$,
\[\lambda = \liminf_{n \to \infty} \frac{1}{n}\log \|\mathcal{L}_{\mathsf{A}}(x,n)\| \leq \max\{\lambda_1,\lambda_2\},\]
and therefore $\lambda=\max\{\lambda_1,\lambda_2\}$ as required. The proof is complete.
\end{proof}
\emph{Remark.} In general the ergodicity hypothesis in Proposition \ref{ergmeasures} may not be removed. We note the following example. Let $\mathcal{I}=\{1,2\}$ and define
\[A_1=\left(\begin{array}{cc}3&0\\0&1\end{array}\right),\qquad A_2=\left(\begin{array}{cc}1&0\\0&3\end{array}\right)\]
and $\mathsf{A}:=\{A_1,A_2\}$. If we define sets $\mathsf{A}^{(1)}$, $\mathsf{A}^{(2)}$ and $\mathsf{B}$ of $1 \times 1$ matrices indexed over $\mathcal{I}$ by $A^{(1)}_1=A^{(2)}_2=3$, $A^{(1)}_2=A^{(2)}_1=1$ and $B_1=B_2=0$, then $\left(\mathsf{A}^{(1)},\mathsf{A}^{(2)},\mathsf{B},I\right)$ defines an upper triangularisation of $\mathsf{A}$. Let $\delta_1$ be the Dirac measure on $\Sigma_2$ concentrated at the constant sequence $y=(y_i)\in\Sigma_2$ whose entries are all equal to $1$, let $\delta_2$ be the Dirac measure concentrated on the constant sequence $z=(z_i)\in\Sigma_2$ whose entries are all equal to $2$, and define $\mu:=\frac{1}{2}(\delta_1+\delta_2)$. An elementary calculation shows that 
\[\inf_{n \geq 1}\frac{1}{n}\int\log \|\mathcal{L}_{\mathsf{A}}(x,n)\|\,d\mu(x) =\log 3,\]
\[ \inf_{n \geq 1}\frac{1}{n}\int\log \|\mathcal{L}_{\mathsf{A}^{(1)}}(x,n)\|\,d\mu(x) = \inf_{n \geq 1}\frac{1}{n}\int\log \|\mathcal{L}_{\mathsf{A}^{(2)}}(x,n)\|\,d\mu(x)=\frac{\log 3}{2}\]
so that the conclusion of Proposition \ref{ergmeasures} is not valid in this case.

\section{Proofs of main theorems}\label{pfs}

We begin with the proofs of Theorem \ref{simple} and Proposition \ref{rabbit}, which are direct consequences of general results on subadditive sequences of functions given in the appendix.

\begin{proof}[of Theorem \ref{simple}]
Since $\|\cdot\|$ and $\vvv\cdot\vvv$ are norms on the finite-dimensional space $\Mat_d(\mathbb{C})$ they are equivalent, and so there exists a constant $C>1$ such that $C^{-1}\|B\| \leq \vvv B \vvv \leq C\|B\|$ for every $B \in \Mat_d(\mathbb{C})$. By \eqref{fart} it follows that
\[\log \varrho(\mathsf{A}) = \lim_{n\to\infty}\sup \left\{\frac{1}{n}\log \left\|\mathcal{L}_{\mathsf{A}}(x,n)\right\| \colon x \in \Sigma_{\mathcal{I}}\right\}=\lim_{n\to\infty}\sup \left\{\frac{1}{n}\log \Lvvv\mathcal{L}_{\mathsf{A}}(x,n)\Rvvv \colon x \in \Sigma_{\mathcal{I}}\right\}.\]
Define a sequence of functions $f_n \colon \Sigma_{\mathcal{I}} \to \mathbb{R} \cup \{-\infty\}$ by $f_n(x):=\log\Lvvv \mathcal{L}_{\mathsf{A}}(x,n)\Rvvv$ for each $n \geq 1$ and $x \in \Sigma_{\mathcal{I}}$. Applying Theorem \ref{StSt} completes the proof of Theorem \ref{simple}.
\end{proof}
\begin{proof}[of Proposition \ref{rabbit}]
Define a sequence of functions $f_n \colon \Sigma_{\mathcal{I}} \to \mathbb{R} \cup \{-\infty\}$ by $f_n(x):=\log\left\| \mathcal{L}_{\mathsf{A}}(x,n)\right\|$ for each $n \geq 1$ and $x \in \Sigma_{\mathcal{I}}$. By definition $\mathcal{M}_{\max}(\mathsf{A})$ is precisely the set of all measures $\mu \in \mathcal{M}_\sigma$ such that $\inf_{n \geq 1}\frac{1}{n}\int f_n\,d\mu = \log \varrho(\mathsf{A})$. By Proposition \ref{mmax} this set is nonempty, compact and convex, and its extreme points are precisely its ergodic elements. If $\vvv \cdot \vvv$ is any norm on $\Mat_d(\mathbb{C})$, choose a constant $C>1$ such that $C^{-1}\|B\| \leq \vvv B \vvv \leq C\|B\|$ for every $B \in \Mat_d(\mathbb{C})$, and define $g_n(x):=\log\Lvvv \mathcal{L}_{\mathsf{A}}(x,n)\Rvvv$ for each $n \geq 1$ and $x \in \Sigma_{\mathcal{I}}$. We have $|f_n(x)-g_n(x)| \leq \log C$ for all $x$ and $n$. If $\mu \in \mathcal{M}_\sigma$ then the sequence $(a_n)$ given by $a_n:=\int f_n\,d\mu$ is subadditive, and so using Lemma \ref{Fuckite} we obtain
\[\inf_{n \geq 1}\frac{1}{n}\int f_n\,d\mu = \lim_{n \to\infty}\frac{1}{n}\int f_n\,d\mu = \lim_{n \to \infty}\frac{1}{n}\int g_n\,d\mu.\]
In particular $\mu \in \mathcal{M}_\sigma$ if and only if $\lim_{n \to \infty}\frac{1}{n}\int \log\Lvvv\mathcal{L}_{\mathsf{A}}(x,n)\Rvvv\,d\mu(x)=\log\varrho(\mathsf{A})$. If additionally $\vvv\cdot\vvv$ is submultiplicative then the sequence $(\int g_n\,d\mu)$ is subadditive, so by Lemma \ref{Fuckite} we have $\lim_{n \to \infty}\frac{1}{n}\int g_n\,d\mu= \inf_{n \geq 1}\frac{1}{n}\int g_n\,d\mu$ in this case. The result follows.
\end{proof}

We now give the much longer proof of Theorem \ref{struc}, starting with the following important special case.

\begin{lemma}\label{rbcase}
Let $\mathsf{A}$ be relatively product bounded. Then parts (i)-(iv) of the conclusions of Theorem \ref{struc} hold for $\mathsf{A}$.\end{lemma}
\begin{proof}
Let $\mathcal{I}$ be an index set for $\mathsf{A}$, and let us denote by $Z$ the set $Z_{\mathsf{A}} \subseteq \Sigma_{\mathcal{I}}$ defined in the statement of Theorem \ref{struc}. We begin by proving (i), for which we recycle a simple argument from \cite[Proposition 1]{M1}.

By Proposition \ref{rabbit}, $\mathcal{M}_{\max}(\mathsf{A})$ is a compact subset of a metrisable topological space, and therefore it is separable. Let $(\mu_n)_{n=1}^\infty$ be a sequence of measures which is dense in $\mathcal{M}_{\max}(\mathsf{A})$, and define $\hat\mu:=\sum_{n=1}^\infty 2^{-n}\mu_n$. Since $\mathcal{M}_{\max}(\mathsf{A})$ is compact and convex we have $\hat\mu \in \mathcal{M}_{\max}(\mathsf{A})$, and therefore $\supp \hat\mu \subseteq Z$. Now, if $U \subset \Sigma_{\mathcal{I}}$ is open with $\nu(U)>0$ for some $\nu \in \mathcal{M}_{\max}(\mathsf{A})$, then the set $\mathcal{U}:=\{\mu \in \mathcal{M}_{\max}(\mathsf{A}) \colon \mu(U)>0\}$ is a nonempty open subset of $\mathcal{M}_{\max}(\mathsf{A})$. Consequently there exists $k \geq 1$ such that $\mu_k \in \mathcal{U}$, and therefore $\hat\mu(U)\geq 2^{-k}\mu_k(U)>0$. We deduce that the open set $\Sigma_{\mathcal{I}} \setminus \supp \hat\mu$ cannot have positive measure with respect to any $\nu \in \mathcal{M}_{\max}(\mathsf{A})$, and it follows that   $Z \subseteq \supp \hat\mu$. We conclude that $Z=\supp\hat\mu$ and part (i) of Theorem \ref{struc} holds for $\mathsf{A}$.
 
We next prove both (ii) and (iv). Since $\mathsf{A}$ is assumed to be relatively product bounded, it has at least one extremal norm. Let $\vvv\cdot\vvv$ be any extremal norm for $\mathsf{A}$, and for each $n \geq 1$ and $x \in \Sigma_{\mathcal{I}}$ define $f_n(x):=\log\vvv\mathcal{L}_{\mathsf{A}}(x,n)\vvv$. Clearly $\sup_{x\in\Sigma_{\mathcal{I}}} f_n(x)=n\log\varrho(\mathsf{A})$ for every $n \geq 1$. It follows by Lemma \ref{subordprin} that the set $Y_{\vvv\cdot\vvv}:=\{x \in \Sigma_{\mathcal{I}} \colon \vvv\mathcal{L}_{\mathsf{A}}(x,n)\vvv=\varrho(\mathsf{A})^n\text{ }\forall\,n \geq 1\}$ is compact and nonempty, satisfies $\sigma Y_{\vvv\cdot\vvv} \subseteq Y_{\vvv\cdot\vvv}$, and has the property that for each measure $\mu \in \mathcal{M}_\sigma$ we have $\mu \in \mathcal{M}_{\max}(\mathsf{A})$ if and only if $\mu\left(Y_{\vvv\cdot\vvv}\right)=1$. Since by part (i) $Z$ is equal to the support of $\hat \mu \in \mathcal{M}_{\max}(\mathsf{A})$, we in particular have $Z \subseteq Y_{\vvv\cdot\vvv}$. It follows that if $\mu \in \mathcal{M}_\sigma$ and $\mu(Z)=1$, then $\mu\left(Y_{\vvv\cdot\vvv}\right)=\mu(Z)=1$ and therefore $\mu \in \mathcal{M}_{\max}(\mathsf{A})$, which establishes (ii). Since $Z\subseteq Y_{\vvv\cdot\vvv}$, and $\vvv\cdot\vvv$ is an arbitrary extremal norm, we immediately deduce (iv).

It remains to prove (iii), for which we modify an argument from \cite{Madv}. Let $\vvv\cdot\vvv$ be an extremal norm for $\mathsf{A}$, and fix any recurrent $x \in Z$. Define
\[\mathcal{S}_{\mathsf{A}}(x):=\bigcap_{m =1}^\infty \left(\overline{\bigcup_{n = m}^\infty \left\{\varrho(\mathsf{A})^{-n}\mathcal{L}_{\mathsf{A}}(x,n) \colon d(\sigma^nx,x)<\frac{1}{m}\right\}}\right).\]
That is, $\mathcal{S}_{\mathsf{A}}(x)$ is the set of all $B \in \mathbf{M}_{d}(\mathbb{C})$ which are equal to the limit as $j \to \infty$ of $\varrho(\mathsf{A})^{-n_j}\mathcal{L}_{\mathsf{A}}(x,n_j)$ along some strictly increasing sequence $(n_j)_{j=1}^\infty$ having the property that $\lim_{j \to \infty}\sigma^{n_j}x=x$. From part (iv) of the Theorem it is clear that $\vvv B\vvv = 1$ for every $B \in \mathcal{S}_{\mathsf{A}}(x)$, and since $\mathcal{S}_{\mathsf{A}}(x)$ is closed it is compact. Since $x$ is recurrent, $\mathcal{S}_{\mathsf{A}}(x)$ is nonempty as a consequence of the compactness of the unit sphere of $\Mat_d(\mathbb{C})$ with respect to $\vvv\cdot\vvv$.  A simple calculation as in \cite{Madv} shows that $\mathcal{S}_{\mathsf{A}}(x)$ is a semigroup with respect to composition of its elements, and since $\mathcal{S}_{\mathsf{A}}(x)$ is compact it follows that it contains an idempotent element $P$ (see e.g. \cite{HM}). The idempotent matrix $P$ is a projection with nonzero norm and therefore satisfies $\rho(P)=1$. It follows from the definition of $\mathcal{S}_{\mathsf{A}}(x)$ that there exists a strictly increasing sequence of integers $(n_j)_{j=1}^\infty$ such that $\varrho(\mathsf{A})^{-n_j}\vvv\mathcal{L}_{\mathsf{A}}(x,n_j)\vvv \to P$, and since the spectral radius functional $\rho \colon \Mat_d(\mathbb{C}) \to \mathbb{R}$ is continuous we have $\lim_{j \to \infty}\varrho(\mathsf{A})^{-n_j}\rho(\mathcal{L}_{\mathsf{A}}(x,n_j)) = 1$. Since clearly $\rho(\mathcal{L}_{\mathsf{A}}(x,n))\leq \vvv\mathcal{L}_{\mathsf{A}}(x,n)\vvv = \varrho(\mathsf{A})^n$ for all $n \geq 1$ this proves (iii).
\end{proof}

We may now prove the Theorem in the general case. We shall prove parts (i) to (iv) of Theorem \ref{struc} by induction on the dimension $d$. If $d=1$ then it is trivially true that $\mathsf{A}$ is relatively product bounded, so the Theorem holds in the case $d=1$ by Lemma \ref{rbcase}.

Let us now assume that conclusions (i) to (iv) of Theorem \ref{struc} have been established for all compact sets $\mathsf{A}\subset \Mat_d(\mathbb{C})$ when $d$ lies in the range $1 \leq d \leq D$, for some natural number $D$. We shall show that these conclusions also hold for all compact sets $\mathsf{A} \subset \Mat_{D+1}(\mathbb{C})$. Let us therefore consider a fixed compact set $\mathsf{A} \subset \mathbf{M}_{D+1}(\mathbb{C})$ indexed over $\mathcal{I}$. If $\mathsf{A}$ is relatively product bounded then the result holds by Lemma \ref{rbcase}, so for the remainder of the proof the induction step we shall assume that $\mathsf{A}$ is not relatively product bounded.

By Proposition \ref{ERXX} there exists an upper triangularisation $\left(\mathsf{A}^{(1)},\mathsf{A}^{(2)},\mathsf{B},M\right)$ of $\mathsf{A}$ such that $\varrho(\mathsf{A})=\varrho\left(\mathsf{A}^{(1)}\right)=\varrho\left(\mathsf{A}^{(2)}\right)$. Let us write $Z_i:=Z_{\mathsf{A}^{(i)}} \subseteq \Sigma_{\mathcal{I}}$ for $i=1,2$ and $Z:=Z_{\mathsf{A}}\subseteq \Sigma_{\mathcal{I}}$, and let $d_1,d_2$ be respectively the dimensions of the matrices comprising $\mathsf{A}^{(1)}$ and $\mathsf{A}^{(2)}$. Since each $d_i$ is less than or equal to $D$, it follows from the induction hypothesis that each $\mathsf{A}^{(i)}$ satisfies properties (i) to (iv) of Theorem \ref{struc}.

The core of the proof of the induction step is to show that $Z = Z_1 \cup Z_2$. 
We begin by showing that if $\mu \in \mathcal{M}_{\max}(\mathsf{A})$, then necessarily $\mu(Z_1 \cup Z_2)=1$. Let $\mathcal{B}$ be the Borel $\sigma$-algebra of the compact metrisable topological space $\mathcal{M}_\sigma$, and choose any $\mu \in \mathcal{M}_{\max}(\mathsf{A})$. By Proposition \ref{simple} $\mathcal{M}_{\max}(\mathsf{A})$ is a compact, convex set whose extremal points are precisely its ergodic elements. By applying a suitable version of Choquet's theorem \cite[p.14]{Phelps} it follows that there exists a measure $\mathbb{P}$ on the measurable space $(\mathcal{M}_{\max}(\mathsf{A}),\mathcal{B})$ such that $\int g\,d\mu=\iint g\,dm \,d\mathbb{P}(m)$ for every continuous function $g \colon \Sigma_{\mathcal{I}} \to \mathbb{R}$, and
\[\mathbb{P}\left(\left\{\nu \in \mathcal{E}_\sigma \colon \inf_{k \geq 1}\frac{1}{k}\int \log\|\mathcal{L}_{\mathsf{A}}(x,k)\|\,d\nu(x) = \log \varrho(\mathsf{A})\right\}\right)=1.\]
By Proposition \ref{ergmeasures} it follows that for $\mathbb{P}$-almost-every $\nu \in \mathcal{M}_\sigma$,
\[\log \varrho(\mathsf{A})=\inf_{k \geq 1}\frac{1}{k}\int \log \left\|\mathcal{L}_{\mathsf{A}}(x,k)\right\|\,d\nu(x) = \max_{i \in \{1,2\}} \inf_{k \geq 1}\frac{1}{k}\int \log \left\|\mathcal{L}_{\mathsf{A}^{(i)}}(x,k)\right\|\,d\nu(x).\]
In particular, since $\varrho(\mathsf{A})=\varrho\left(\mathsf{A}^{(1)}\right)=\varrho\left(\mathsf{A}^{(2)}\right)$ we conclude that $\mathbb{P}$-almost-every $\nu \in \mathcal{M}_\sigma$ belongs to  $\mathcal{M}_{\max}\left(\mathsf{A}^{(1)}\right) \cup \mathcal{M}_{\max}\left(\mathsf{A}^{(2)}\right)$. It follows by the definition of $Z_1$ and $Z_2$ that $\nu(Z_1 \cup Z_2)=1$ for $\mathbb{P}$-a.e. $\nu \in \mathcal{M}_\sigma$. Since $Z_1 \cup Z_2$ is closed and $\Sigma_{\mathcal{I}}$ is metrisable, we may choose a non-negative continuous function $g \colon \Sigma_{\mathcal{I}} \to \mathbb{R}$ which is zero precisely on $Z_1 \cup Z_2$. We then have $\int g\,d\nu=0$ for $\mathbb{P}$-almost-every $\nu \in \mathcal{M}_{\sigma}$, and it follows by integration that $\int g\,d\mu = \iint g\,d\nu\, d\mathbb{P}(\nu)=0$.  We conclude that $\mu(Z_1 \cup Z_2)=1$ as claimed. Since $\mu$ is an arbitrary element of $\mathcal{M}_{\max}(\mathsf{A})$, we furthermore obtain the result that $Z \subseteq Z_1 \cup Z_2$.

We now conversely claim that if $\mu \in \mathcal{M}_\sigma$ and $\mu(Z_1 \cup Z_2)=1$, then $\mu \in\mathcal{M}_{\max}(\mathsf{A})$. Let us fix such a measure $\mu$. It is a classical result that the set $\mathcal{M}_\sigma$ is compact and convex, with its set of extremal points being precisely equal to $\mathcal{E}_\sigma$ (see e.g. \cite[p.152]{W}). Applying Choquet's theorem we find that there exists a probability measure $\mathbb{P}$ on the measurable space $(\mathcal{M}_\sigma,\mathcal{B})$ such that $\int g\,d\mu=\iint g\,dm\,d\mathbb{P}(m)$ for every continuous $g \colon \Sigma_{\mathcal{I}} \to \mathbb{R}$, and $\mathbb{P}(\mathcal{E}_\sigma)=1$. Considering again a non-negative continuous function $g$ which is zero precisely on $Z_1 \cup Z_2$, we note that since $\int g\,d\mu=0$ we must have $\int g\,d\nu=0$ for $\mathbb{P}$-a.e. $\nu \in\mathcal{M}_\sigma$, and therefore $\nu(Z_1\cup Z_2)=1$ for $\mathbb{P}$-almost-every $\nu$. Since the sets $Z_1$ and $Z_2$ are invariant, an ergodic measure which gives full measure to $Z_1\cup Z_2$ must also give full measure to at least one of $Z_1$ and $Z_2$. By the induction hypothesis part (ii) of Theorem 2.3 applies to each $\mathsf{A}^{(i)}$, and we deduce that $\mathbb{P}$-almost-every measure $\nu \in\mathcal{M}_\sigma$ belongs to at least one of $\mathcal{M}_{\max}\left(\mathsf{A}^{(1)}\right)$ and $ \mathcal{M}_{\max}\left(\mathsf{A}^{(2)}\right)$. Applying Proposition \ref{ergmeasures} once more we deduce that $\mathbb{P}(\mathcal{M}_{\max}(\mathsf{A}))=1$. For each fixed pair of integers $n,k \geq 1$, the function $x \mapsto \max\{\log\|\mathcal{L}_{\mathsf{A}}(x,n)\|,-k\}$ is a continuous function from $\Sigma_{\mathcal{I}}$ to $\mathbb{R}$, and therefore
\begin{align*}\frac{1}{n}\int\max\{\log\|\mathcal{L}_{\mathsf{A}}(x,n)\|,-k\}\,d\mu(x) &= \frac{1}{n}\iint\max\{\log\|\mathcal{L}_{\mathsf{A}}(x,n)\|,-k\}\,d\nu(x)d\mathbb{P}(\nu) \\
&\geq \frac{1}{n}\iint\log\|\mathcal{L}_{\mathsf{A}}(x,n)\|\,d\nu(x)d\mathbb{P}(\nu) \geq\log\varrho(\mathsf{A}).\end{align*}
Letting $k \to \infty$ for each fixed $n$, applying the monotone convergence theorem and then taking the infimum over $n$ we obtain $\inf_{n\geq 1} \frac{1}{n}\int \log\|\mathcal{L}_{\mathsf{A}}(x,n)\|\,d\mu(x) \geq \log\varrho(\mathsf{A})$, and we conclude that $\mu \in \mathcal{M}_{\max}(\mathsf{A})$ as claimed.

We can now show that $Z=Z_1 \cup Z_2$. By the induction hypothesis, property (i) of the Theorem holds for $\mathsf{A}^{(1)}$ and $\mathsf{A}^{(2)}$,  and hence for $i=1,2$ there exists $\hat\mu_i \in \mathcal{M}_{\max}\left(\mathsf{A}^{(i)}\right)$ such that $\supp \hat\mu_i=Z_i$. The $\sigma$-invariant measure $\hat\mu := \frac{1}{2}\left(\hat\mu_1+\hat\mu_2\right)$ has support equal to $Z_1\cup Z_2$, and as previously noted this implies that $\hat\mu \in\mathcal{M}_{\max}(\mathsf{A})$. We deduce that $Z_1 \cup Z_2 \subseteq Z$ and therefore $Z=Z_1 \cup Z_2$ as claimed. Combining this observation with the results already established, it follows immediately that parts (i) and (ii) of Theorem \ref{struc} hold for $\mathsf{A}$. 

To complete the proof of the induction step it remains to show that parts (iii) and (iv) of the Theorem hold for $\mathsf{A}$. We begin with part (iii). It is a straightforward consequence of the relation
\[M^{-1}\mathcal{L}_{\mathsf{A}}(x,n)M = \left(\begin{array}{cc}\mathcal{L}_{\mathsf{A}^{(1)}}(x,n)&*\\0&\mathcal{L}_{\mathsf{A}^{(2)}}(x,n)\end{array}\right)\]
together with Gelfand's formula that $\rho(\mathcal{L}_{\mathsf{A}}(x,n))=\max_{i=1,2}\rho\left(\mathcal{L}_{\mathsf{A}^{(i)}}(x,n)\right)$ for all $x \in \Sigma_{\mathcal{I}}$ and $n \geq 1$. By the induction hypothesis we know that for each $i$, every $x \in Z_i$ satisfies $\limsup_{n \to \infty}\varrho\left(\mathsf{A}^{(i)}\right)^{-n}\rho(\mathcal{L}_{\mathsf{A}^{(i)}}(x,n))=1$, and since $Z=Z_1\cup Z_2$ it follows directly that $\limsup_{n \to \infty}\varrho(\mathsf{A})^{-n}\rho(\mathcal{L}_{\mathsf{A}}(x,n))=1$ for every $x \in Z$. Similarly, by the induction hypothesis every $x \in Z_1\cup Z_2=Z$ satisfies $\inf_{n \geq 1} \max_{i=1,2}\varrho(\mathsf{A})^{-n}|\mathcal{L}_{\mathsf{A}^{(i)}}(x,n)|>0$, and it follows that every $x \in Z$ satsifies $\inf_{n \geq 1}\varrho(\mathsf{A})^{-n}\|\mathcal{L}_{\mathsf{A}}(x,n)\|>0$ as required to prove (iv). We have shown that if parts (i)-(iv) of Theorem \ref{struc} hold for all compact sets $\mathsf{A}$ of matrices of dimension at most $D$, then they also hold for all compact sets of matrices of dimension at most $D+1$. It follows by induction that properties (i)-(iv) of Theorem \ref{struc} hold for every $\mathsf{A}$ within the scope of the statement of the Theorem.

To complete the proof of the Theorem it remains only to establish (v). Let us fix $\mathsf{A} \subset \Mat_d(\mathbb{C})$ indexed in $\mathcal{I}$, and let $x\in \Sigma_{\mathcal{I}}$ be weakly extremal. For each $n \geq 1$ define $\mu_n:=(1/n)\sum_{k=0}^{n-1}\delta_{\sigma^kx} \in \mathcal{M}$. Applying Lemma \ref{wkx} with $f_n(x):=\|\mathcal{L}_{\mathsf{A}}(x,n)\|$ we see that every accumulation point of $(\mu_n)$ lies in $\mathcal{M}_{\max}(\mathsf{A})$. Define a continuous function $g \colon \Sigma_{\mathcal{I}} \to \mathbb{R}$ by $g(x):=\mathrm{dist}(x,Z_{\mathsf{A}})$. If $(\mu_{n_j})$ is a convergent subsequence of $(\mu_n)$ with limit $\mu \in \mathcal{M}_{\max}(\mathsf{A})$ then clearly $\int g\,d\mu=0$ using the definition of $Z_{\mathsf{A}}$. It follows that $(1/n_j)\sum_{k=0}^{n_j-1}\mathrm{dist}(\sigma^kx,Z_{\mathsf{A}}) \to 0$ along any subsequence $(n_j)$ such that the sequence of measures $(\mu_{n_j})$ is  convergent in $\mathcal{M}$. Since $\mathcal{M}$ is compact and metrisable, it follows that every subsequence $(n_j)$ of the natural numbers possesses a finer subsequence with this property, and we conclude that $(1/n)\sum_{k=0}^{n-1}\mathrm{dist}(\sigma^kx,Z_{\mathsf{A}}) \to 0$ as required. The proof of the Theorem is complete.

\section{Applications to stability theory for discrete linear inclusions}\label{appmarkov}

Given a finite set $\mathsf{A}=\{A_1,\ldots,A_\ell\}\subset \Mat_d(\mathbb{C})$, the discrete linear inclusion associated to $\mathsf{A}$ is the set of all possible sequences of vectors $(v_n)_{n \geq 0}$ belonging to $\mathbb{C}^d$ having the form $v_n \equiv \mathcal{L}_{\mathsf{A}}(x,n)v_0$ for some $x \in \Sigma_\ell$. In the past two decades discrete linear inclusions have attracted substantial interest in the control theory literature \cite{Ba,BCS,DHM,FMC,Gurvits,SWMWK,SG,Wirth-1}.

In the influential article \cite{Gurvits}, L. Gurvits introduced the following notions of stability for discrete linear inclusions. The discrete linear inclusion corresponding to $\mathsf{A}$ is said to be \emph{absolutely asymptotically stable} if $v_n \to 0$ for every trajectory of the inclusion, or equivalently if $\mathcal{L}_{\mathsf{A}}(x,n) \to 0$ for every $x \in \Sigma_\ell$. We say that the inclusion associated to $\mathsf{A}$ is \emph{periodically asymptotically stable} if $\mathcal{L}_{\mathsf{A}}(x,n)v_0 \to 0$ for all $v_0 \in \mathbb{C}^d$ whenever $x \in \Sigma_\ell$ is a periodic sequence, or equivalently if $\mathcal{L}_{\mathsf{A}}(x,n) \to 0$ for every periodic $x \in\Sigma_\ell$. Finally, the inclusion associated to $\mathsf{A}$ is \emph{Markov asymptotically stable} if for every Markov measure $\mu \in \mathcal{M}$ for which all transition probabilities are nonzero, we have $\mathcal{L}_{\mathsf{A}}(x,n)v_0 \to 0$ $\mu$-a.e. for all $v_0 \in \mathbb{C}^d$, or equivalently $\mathcal{L}_{\mathsf{A}}(x,n)\to 0$ $\mu$-a.e. In the remainder of this section we shall simply say that $\mu$ is a \emph{full Markov measure} if it is a Markov measure on $\Sigma_\ell$ such that all transition probabilities in the associated transition matrix are nonzero. Note that every full Markov measure belongs to $\mathcal{E}_\sigma$.

It is not difficult to see that the linear inclusion associated to $\mathsf{A}$ is absolutely asymptotically stable if and only if $\varrho(\mathsf{A})<1$, and moreover that absolute asymptotic stability implies both periodic and Markov asymptotic stability. L. Gurvits asked in \cite{Gurvits} whether periodic asymptotic stability implies absolute asymptotic stability; the existence of counterexamples to the Lagarias-Wang finiteness conjecture (see section \ref{appratio} below) answered this question in the negative.  On the other hand it is straightforward to exhibit sets of matrices whose associated discrete linear inclusions are Markov asymptotically stable but not absolutely asymptotically stable, for example the set
\[\mathsf{A}:=\left\{\left(\begin{array}{cc}0&\lambda\\0&0\end{array}\right),\left(\begin{array}{cc}0&0\\\lambda&0\end{array}\right)\right\}\]
has this property for any $\lambda\geq 1$. In this section we give a proof of the following result, which was recently announced by X. Dai, Y. Huang and M. Xiao in conference proceedings \cite{DHM2}:
\begin{theorem}\label{mkv}
Let $\mathsf{A}=\{A_1,\ldots,A_\ell\} \subset \Mat_d(\mathbb{C})$, and suppose that the associated discrete linear inclusion is periodically asymptotically stable. Then it is also Markov asymptotically stable.
\end{theorem}
We shall deduce Theorem \ref{mkv} from the following simple lemma:
\begin{lemma}\label{pasmas}
Let $\mathsf{A}=\{A_1,\ldots,A_\ell\} \subset \Mat_d(\mathbb{C})$ be a finite set of matrices such that the associated discrete linear inclusion is not Markov asymptotically stable. Then either $\varrho(\mathsf{A})>1$, or $\varrho(\mathsf{A})=1$ and $Z_{\mathsf{A}}=\Sigma_\ell$.
\end{lemma}
\begin{proof}
If $\mu \in \mathcal{E}_\sigma$ is any full Markov measure, then by the subadditive ergodic theorem (given in the appendix as Theorem \ref{SAET}) we have for $\mu$-almost-every $x \in \Sigma_\ell$
\begin{equation}\label{obvious} \lim_{n \to \infty} \frac{1}{n}\log\|\mathcal{L}_{\mathsf{A}}(x,n)\|= \inf_{n \geq 1}\frac{1}{n}\int \log\|\mathcal{L}_{\mathsf{A}}(z,n)\|\,d\mu(z) \leq \log\varrho(\mathsf{A}).\end{equation}
If the discrete linear inclusion associated to $\mathsf{A}$ is not Markov asymptotically stable then there exists a full Markov measure $\mu \in \mathcal{E}_\sigma$ such that $\|\mathcal{L}_{\mathsf{A}}(x,n)\|$ does not converge to zero $\mu$-a.e. in the limit as $n \to \infty$. Using \eqref{obvious} we deduce that $\inf_{n \geq 1}\frac{1}{n}\int \log\|\mathcal{L}_{\mathsf{A}}(x,n)\|\,d\mu(x) \geq 0$ and therefore $\varrho(\mathsf{A}) \geq 1$. If $\varrho(\mathsf{A})>1$ then there is nothing left to prove, so we shall assume that $\varrho(\mathsf{A})=1$. In this case the measure $\mu$ belongs to $\mathcal{M}_{\max}(\mathsf{A})$, and hence $\supp \mu$ is contained in $Z_{\mathsf{A}}$. Since every full Markov measure has support equal to $\Sigma_\ell$ we conclude that $Z_{\mathsf{A}}=\Sigma_\ell$ as claimed.
\end{proof}

\begin{proof}[of Theorem \ref{mkv}] Suppose that the discrete linear inclusion associated to $\mathsf{A}$ is periodically asymptotically stable. We may deduce from Theorem \ref{struc}(iii) that
\[\varrho(\mathsf{A})= \sup_{n \geq 1}\left\{\rho(A_{i_1}\cdots A_{i_n})^{\frac{1}{n}} \colon (i_1,\ldots,i_n)\in\{1,\ldots,\ell\}^n\right\},\]
an identity originally due to M. A. Berger and Y. Wang \cite{BW}. It follows from this identity together with periodic asymptotic stablility that $\varrho(\mathsf{A}) \leq 1$. If $\varrho(\mathsf{A})<1$ then the associated discrete linear inclusion is absolutely asymptotically stable and hence is Markov asymptotically stable. If on the other hand $\varrho(\mathsf{A})=1$, then it follows from Theorem \ref{struc}(iv) together with periodic asymptotic stability that the Mather set $Z_{\mathsf{A}}$ does not contain any periodic orbits. In particular $Z_{\mathsf{A}} \neq \Sigma_\ell$, and it follows by Lemma \ref{pasmas} that the discrete linear inclusion associated to $\mathsf{A}$ is Markov asymptotically stable. The proof is complete.
\end{proof}

\section{Applications to the study of the 1-ratio of pairs of matrices}\label{appratio}

In the influential article \cite{LW}, J. Lagarias and Y. Wang posed the following question: if $\mathsf{A}=\{A_1,\ldots,A_\ell\}$ is a finite set of real $d \times d$ matrices, is it always the case that there exist an integer $n \geq 1$ and indices $i_1,\ldots,i_n \in \{1,\ldots,\ell\}$ such that $\rho(A_{i_1}\cdots A_{i_n})=\varrho(\mathsf{A})^n$? This question was answered negatively in 2002 by T. Bousch and J. Mairesse \cite{BM}, but continues to stimulate research \cite{BTV,CGSC,HMST,JB,Koz3,LS}.

Let us say that $\mathsf{A}$ has the \emph{finiteness property} if Lagarias and Wang's condition holds for $\mathsf{A}$. Examples of pairs of $2 \times 2$ matrices which do not satisfy the finiteness property were constructed in \cite{BTV,BM,HMST,Koz3} by a variety of different methods; the latter three of these four articles possess some important common features which we now describe. In each case one studies a family of pairs of $2 \times 2$ matrices $\mathsf{A}_\lambda = \{A_1(\lambda),A_2(\lambda)\}$ depending continuously on a parameter $\lambda$ which belongs to a connected subregion $\Delta$ of some $\mathbb{R}^k$. It is then shown that for each $\lambda$ in this region, if $x \in \Sigma_2$ has the property that $\mathcal{L}_{\mathsf{A}_\lambda}(x,n)$ grows rapidly in some suitable sense as $n \to \infty$, then the sequence $x$ has a well-defined proportion of $1$'s depending only on $\lambda$: that is, in the limit as $n \to \infty$ the quantity $\frac{1}{n}\#\{ 1 \leq j \leq n\colon x_j=1\}$ converges to some value $\mathfrak{r}(\lambda)\in[0,1]$. It is then shown that the $1$-ratio function $\mathfrak{r}\colon \Delta \to [0,1]$ is continuous and nonconstant. Since $\Delta$ is connected this allows us to deduce that there exists $\lambda_0 \in \Delta$ such that $\mathfrak{r}(\lambda_0)$ is irrational. It follows that $\mathsf{A}_{\lambda_0}$ cannot have the finiteness property, since if it were the case that $\rho(A_{i_m}(\lambda_0)\cdots A_{i_1}(\lambda_0))=\varrho(\mathsf{A}_{\lambda_0})^m$, then the periodic sequence $x \in \Sigma_2$ defined by $x_{km+j}:=i_j$ for all $k \geq 0$ and $1 \leq j \leq m$ would have the property that $\mathcal{L}_{\mathsf{A}_{\lambda_0}}(x,n)$ grows rapidly, and therefore $\lim_{n \to \infty}\frac{1}{n}\#\{ 1 \leq j \leq n\colon x_j=1\} =\mathfrak{r}(\lambda_0)\notin \mathbb{Q}$; but by periodicity $\mathfrak{r}(\lambda_0)$ must be rational with denominator not greater than $m$, yielding a contradiction.

We remark that the counterexample to the Lagarias-Wang finiteness conjecture given by Blondel, Theys and Vladimirov \cite{BTV} is currently the only counterexample in the literature which does not make direct use of the approach given above. However, a detailed investigation by J. Theys \cite{Tthesis} shows that a continuous $1$-ratio also exists for the family of pairs of matrices used in \cite{BTV} to construct a counterexample. This was exploited by Hare, Sidorov, Theys and the present author in \cite{HMST} to exhibit an \emph{explicit} pair of matrices which fails to have the finiteness property.

A major point of difference between the articles \cite{BM,HMST,Koz3} is the precise definition which is used to express the property that $\mathcal{L}_{\mathsf{A}}(x,n)$ ``grows rapidly''. In the examples considered by Bousch and Mairesse, it is shown that for each of the pairs of matrices $\mathsf{A}$ being considered, there is a unique Borel probability measure on $\Sigma_2$ which is invariant under the shift and which maximises the a.e. pointwise limit of $\frac{1}{n}\log\|\mathcal{L}_{\mathsf{A}}(x,n)\|$; in our formalism this corresponds to showing that $\mathcal{M}_{\max}(\mathsf{A})$ contains a unique measure. The concept of rapidly growing sequence therefore corresponds to that of a measure-theoretically typical sequence with respect to any maximising measure. In the article \cite{HMST}, a sequence is considered to be ``rapidly growing'' for the purpose of defining the $1$-ratio if it is weakly extremal in the sense defined in \S2. Finally, in the article \cite{Koz3}, the notion of rapidly growing sequence $x \in \Sigma_2$ used by Kozyakin is that there exists a Barabanov norm $\vvv\cdot \vvv$ for $\mathsf{A}$ and a vector $v \in \mathbb{C}^d$ such that $\varrho(\mathsf{A})^{-n}\vvv \mathcal{L}_{\mathsf{A}}(x,n)v\vvv = \vvv v \vvv \neq 0$ for every $n \geq 1$. One of the objectives of this section is to show that the different definitions of $1$-ratio described above yield equivalent results; this is undertaken in Proposition \ref{e1r}. We thus give a general and unified definition of the $1$-ratio of a finite set of square matrices.

A second point in common between the articles \cite{BM,HMST,Koz3} is that the methods used to prove the continuity of the function $\mathfrak{r}$ do not easily generalise to broader classes of matrices. While the ergodic-theoretic approach of Bousch and Mairesse is relatively general, the context of their theorems imposes strict positivity hypotheses on the matrices comprising the set $\mathsf{A}$. The arguments in \cite{HMST} and \cite{Koz3} are much more highly specialised, and seem unlikely to be applicable beyond the case of pairs of real $2 \times 2$ matrices, with one matrix being upper triangular and the other lower triangular. The second objective of this section is to show in great generality that if a well-defined $1$-ratio exists for each member of a family of finite sets of square matrices, then it of necessity varies continuously within that family.

Given integers $i,\ell$ with $1 \leq i \leq \ell$, we let $[i]$ denote the set of all sequences $x \in \Sigma_\ell$ whose first entry is equal to $i$. For every $x \in\Sigma_{\mathcal{I}}$ and $n \geq 1$ we thus have $\sum_{k=0}^{n-1}\chi_{[i]}(\sigma^kx)=\#\{1 \leq j \leq n \colon x_j=i\}$. Note that $[i]$ is both open and closed as a subset of $\Sigma_\ell$. Given a set $\mathsf{A}=\{A_1,\ldots,A_\ell\}\subset \Mat_d(\mathbb{C})$, let us say that $x \in \Sigma_\ell$ is \emph{Kozyakin extremal} if there exist a Barabanov norm $\vvv\cdot\vvv$ for $\mathsf{A}$ and a vector $v \in \mathbb{C}^d$ with $\vvv v \vvv=1$ such that $\vvv\mathcal{L}_{\mathsf{A}}(x,n)v\vvv=\varrho(\mathsf{A})^n$ for all $n \geq 1$.
The following proposition shows that the definitions of $1$-ratio used by Bousch-Mairesse in \cite{BM}, by Hare-Morris-Sidorov-Theys in \cite{HMST} and by Kozyakin in \cite{Koz3} are equivalent:
\begin{proposition}\label{e1r}
Let $\mathsf{A}=\{A_1,\ldots,A_\ell\} \subset \Mat_d(\mathbb{C})$, and let $\gamma \in [0,1]$ and $1 \leq i \leq \ell$. Then the following properties are equivalent:
\begin{enumerate}
\item
For every weakly extremal orbit $x \in \Sigma_{\ell}$ we have $\frac{1}{n}\sum_{k=0}^{n-1}\chi_{[i]}(\sigma^kx) \to \gamma$.
\item
For every strongly extremal orbit $x \in \Sigma_{\ell}$ we have $\frac{1}{n}\sum_{k=0}^{n-1}\chi_{[i]}(\sigma^kx) \to \gamma$.
\item
For every ergodic $\mu \in \mathcal{M}_{\max}(\mathsf{A})$ we have $\mu([i])=\gamma$.
\item
For every $\mu \in \mathcal{M}_{\max}(\mathsf{A})$ we have $\mu([i])=\gamma$.
\end{enumerate}
Suppose additionally that there exists a Barabanov norm for $\mathsf{A}$. Then each of properties (i)-(iv) is additionally equivalent to the following property: for every Kozyakin extremal $x \in \Sigma_\ell$ , we have $\frac{1}{n}\sum_{k=0}^{n-1}\chi_{[i]}(\sigma^kx) \to \gamma$.
\end{proposition}
\begin{proof}
Since every strongly extremal orbit is weakly extremal it is clear that (i) implies (ii). If $\mathsf{A}$ satisfies (ii), then by part (iv) of Theorem \ref{struc} it follows that for every $\mu \in \mathcal{M}_{\max}(\mathsf{A})$ we have $\lim_{n \to \infty} \frac{1}{n}\sum_{k=0}^{n-1}\chi_{[i]}(\sigma^kx) = \gamma$ for $\mu$-a.e. $x$, and hence by the Birkhoff ergodic theorem we have $\mu([i])=\gamma$ for every ergodic $\mu \in \mathcal{M}_{\max}(\mathsf{A})$, which gives (iii).

If (iii) holds, let us choose any $\mu \in \mathcal{M}_{\max}(\mathsf{A})$. Using Proposition \ref{rabbit} and Choquet's theorem as in previous sections, we may find a measure $\mathbb{P}$ on the compact convex set $\mathcal{M}_{\max}(\mathsf{A})$ such that $\mathbb{P}(\mathcal{M}_{\max}(\mathsf{A}) \cap \mathcal{E}_\sigma) = 1$ and $\int g\,d\mu=\iint g\,dm\,d\mathbb{P}(m)$ for all continuous functions $g \colon \Sigma_{
\ell} \to \mathbb{R}$. The function $\chi_{[i]}$ is continuous, and since by hypothesis $\int \chi_{[i]}\,dm=\gamma$ for $\mathbb{P}$-almost-every $m$ we have $\mu([i])=\gamma$ by integration. Since $\mu$ is arbitrary we conclude that (iv) holds.

Now let us suppose that $\mathsf{A}$ satisfies (iv). Let $x\in\Sigma_\ell$ be weakly extremal, and for each $n \geq 1$ define a Borel probability measure on $\Sigma_\ell$ by $\mu_n:=\frac{1}{n}\sum_{k=0}^{n-1}\delta_{\sigma^kx}$, where $\delta_z$ denotes the Dirac probability measure concentrated at $z$. Choose any increasing sequence of integers $(n_j)$ with the property that $\frac{1}{n_j}\sum_{k=0}^{n_j-1} \chi_{[i]}(\sigma^kx)$ converges to some limit; we will show that the only possibly value of this limit is $\gamma$, which implies (i). By replacing $(n_j)$ with a subsequence if necessary, using the fact that $\mathcal{M}$ is compact and metrisable, we may assume that $\mu_{n_j} \to \mu$ for some measure $\mu \in \mathcal{M}$. Since $x$ is weakly extremal, we may apply Lemma \ref{wkx} with $f_n(x):=\log\|\mathcal{L}_{\mathsf{A}}(x,n)\|$ to deduce that $\mu\in\mathcal{M}_\sigma$ and $\inf_{m\geq 1}\frac{1}{m}\int \log\|\mathcal{L}_{\mathsf{A}}(x,m)\|\,d\mu(x) = \log\varrho(\mathsf{A})$ and therefore $\mu \in \mathcal{M}_{\max}(\mathsf{A})$. Since $\mu_{n_j} \to \mu$ and the function $\chi_{[i]} \colon \Sigma_\ell \to \mathbb{R}$ is continuous, it follows that $\lim_{j \to \infty} \frac{1}{n_j}\sum_{k=0}^{n_j-1} \chi_{[i]}(\sigma^kx)=\lim_{n\to\infty} \int \chi_{[i]}\,d\mu_n=\int \chi_{[i]}\,d\mu = \mu([i])=\gamma$. This completes the proof of the equivalence of statements (i)-(iv).

Finally, let us suppose that a Barabanov norm for $\mathsf{A}$ exists. We claim that every $x \in Z_{\mathsf{A}}$ is Kozyakin extremal. Choose any Barabanov norm $\vvv\cdot\vvv$ for $\mathsf{A}$, and let $x \in Z_{\mathsf{A}}$. By part (iv) of Theorem \ref{struc} we have $\vvv \mathcal{L}_{\mathsf{A}}(x,n)\vvv=\varrho(\mathsf{A})^n$ for all $n \geq 1$. It follows that there exists a subsequence $(n_j)$ of the natural numbers such that $\varrho(\mathsf{A})^{-n_j}\mathcal{L}_{\mathsf{A}}(x,n_j)$ converges as $j \to \infty$ to some matrix $L$ such that $\vvv L \vvv=1$. Choose a vector $v \in \mathbb{C}^d$ such that $\vvv v \vvv=\vvv Lv \vvv = 1$. Fix any $n \geq 1$ and let $j$ be sufficiently large that $n_j>n$. We have
\[\varrho(\mathsf{A})^{-n_j} \vvv\mathcal{L}_{\mathsf{A}}(x,n_j)v\vvv \leq \varrho(\mathsf{A})^{-n_j}\vvv\mathcal{L}_{\mathsf{A}}(\sigma^nx,n_j-n)\vvv. \vvv\mathcal{L}_{\mathsf{A}}(x,n)v\vvv \leq \varrho(\mathsf{A})^{-n} \vvv\mathcal{L}_{\mathsf{A}}(x,n)v\vvv \leq 1,\]
and since $\lim_{j \to \infty} \varrho(\mathsf{A})^{-n_j} \vvv\mathcal{L}_{\mathsf{A}}(x,n_j)v\vvv = \vvv Lv \vvv=1$ we deduce that $\vvv \mathcal{L}_{\mathsf{A}}(x,n)v\vvv=\varrho(\mathsf{A})^n$. It follows that $x$ is Kozyakin extremal as claimed. Now, if (ii) holds, then since every Kozyakin extremal point is strongly extremal, we have $\frac{1}{n}\sum_{k=0}^{n-1}\chi_{[i]}(\sigma^kx) \to \gamma$ for every Kozyakin extremal $x \in \Sigma_\ell$. Conversely, if $\frac{1}{n}\sum_{k=0}^{n-1}\chi_{[i]}(\sigma^kx) \to \gamma$ for every $x \in \Sigma_\ell$ which is Kozyakin extremal, then in particular $\frac{1}{n}\sum_{k=0}^{n-1}\chi_{[i]}(\sigma^kx) \to \gamma$ for every $x \in Z_{\mathsf{A}}$. It follows that $\frac{1}{n}\sum_{k=0}^{n-1}\chi_{[i]}(\sigma^kx) \to \gamma$ $\mu$-a.e. for every ergodic $\mu \in \mathcal{M}_{\max}(\mathsf{A})$, and by the Birkhoff ergodic theorem this implies that (iii) holds. This completes the proof of the proposition.
\end{proof}

Given a finite set $\mathsf{A}=\{A_1,\ldots,A_\ell\}\subset \Mat_d(\mathbb{C})$ and an integer $1 \leq i \leq \ell$, we shall say that $\mathsf{A}$ has a \emph{unique optimal} $i$-\emph{ratio} if there exists a real number $\gamma \in [0,1]$ such that one of the conditions in Proposition \ref{e1r} is satisfied. As previously mentioned, some examples of pairs of matrices for which a unique optimal $1$-ratio exists are given in \cite{BM,HMST,Koz3}. Note that Proposition \ref{e1r} implies that for any $\ell \geq 2$, a sufficient condition for $\mathsf{A}=\{A_1,\ldots,A_\ell\}$ to have a unique optimal $1$-ratio is that $\mathcal{M}_{\max}(\mathsf{A})$ contains a unique measure. Conversely, if $\ell \geq 2$ and $Z_{\mathsf{A}} = \Sigma_\ell$ then it is clear that $\mathsf{A}$ does not have a unique optimal $i$-ratio for any $i$; this situation arises for example if $\mathsf{A}$ consists entirely of isometries of $\mathbb{R}^d$. It is interesting to ask whether a `typical' finite set of $d \times d$ complex matrices admits a unique optimal $1$-ratio. In view of Proposition \ref{e1r}, it would be sufficient to show that for a typical finite set of complex $d \times d$ matrices $\mathsf{A}$, the set $\mathcal{M}_{\max}(\mathsf{A})$ is a singleton set.

The following general result shows that when $i$-ratios exist, they are continuous:
\begin{theorem}\label{1rc}
Let $\Delta$ be a metric space, let $A_1,\ldots,A_\ell \colon \Delta \to \Mat_d(\mathbb{C})$ be continuous functions, and let $1 \leq i \leq \ell$. Suppose that for each $\lambda \in \Delta$ the set $\mathsf{A}_\lambda := \{A_1(\lambda),\ldots,A_\ell(\lambda)\}$ has a unique optimal $i-ratio$, which we denote by $\mathfrak{r}(\lambda)$. Then $\mathfrak{r} \colon \Delta \to [0,1]$ is continuous.
\end{theorem}
In order to prove Theorem \ref{1rc} we shall use the following general lemma.
\begin{lemma}\label{technobanana}
Let $\mathsf{A} \subset \Mat_d(\mathbb{C})$  be a compact set indexed in $\mathcal{I}$. For each $n \geq 1$ let $\mathsf{A}_n = \left\{A_i^{(n)} \colon i \in \mathcal{I}\right\}$ be a compact subset of $\Mat_d(\mathbb{C})$ indexed in $\mathcal{I}$, and suppose that the uniform difference $\sup\left\{\left\|A_i-A^{(n)}_i\right\| \colon i \in \mathcal{I}\right\}$ converges to zero as $n \to \infty$. If $\mu_n \in \mathcal{M}_{\max}(\mathsf{A}_n)$ for each $n \geq 1$, and $\mu \in \mathcal{M}_\sigma$ is a limit point of $(\mu_n)$, then $\mu \in \mathcal{M}_{\max}(\mathsf{A})$. 
\end{lemma}
\begin{proof}
We shall make use of the fact that the convergence hypothesis on $\mathsf{A}_n$ implies that $\lim_{n \to \infty}\varrho(\mathsf{A}_n) =\varrho(\mathsf{A})$. For a proof of this statement see e.g. \cite{Koz4,Wirth1}. By passing to an appropriate subsequence if necessary, we may assume that in fact $\mu$ is the weak-* limit of the sequence $(\mu_n)$.

Let us fix two integers $k,m \geq 1$. It follows from our hypotheses that the sequence of continuous functions from $\Sigma_{\mathcal{I}}$ to $\mathbb{R}$ defined by $x \mapsto \max\{\log \|\mathcal{L}_{\mathsf{A}_n}(x,m)\|,-k\}$ converges uniformly as $n \to \infty$ to the continuous function defined by $x \mapsto \max\{\log \|\mathcal{L}_{\mathsf{A}}(x,m)\|,-k\}$, and hence by a uniform estimate we may establish the equality
\[\lim_{n\to\infty}\left|\frac{1}{m}\int \left(\max\{\log \|\mathcal{L}_{\mathsf{A}_n}(x,m)\|,-k\}- \max\{\log \|\mathcal{L}_{\mathsf{A}}(x,m)\|,-k\}\right)d\mu_n\right|=0.\]
On the other hand, since $(\mu_n)$ converges weakly to $\mu$, we have
\[\lim_{n\to\infty}\left|\frac{1}{m}\int \max\{\log \|\mathcal{L}_{\mathsf{A}}(x,m)\|,-k\}d\mu_n- \frac{1}{m}\int\max\{\log \|\mathcal{L}_{\mathsf{A}}(x,m)\|,-k\}d\mu\right|=0\]
and combining these two estimates we obtain
\[\lim_{n \to \infty}\frac{1}{m}\int \max\{\log \|\mathcal{L}_{\mathsf{A}_n}(x,m)\|,-k\}d\mu_n = \frac{1}{m}\int \max\{\log \|\mathcal{L}_{\mathsf{A}}(x,m)\|,-k\}d\mu\]
for each $m$ and $k$. Hence for each $m,k \geq 1$,
\begin{align*}\frac{1}{m}\int \max\{\log \|\mathcal{L}_{\mathsf{A}}(x,m)\|,-k\}d\mu&=\lim_{n \to \infty}\frac{1}{m}\int \max\{\log \|\mathcal{L}_{\mathsf{A}_n}(x,m)\|,-k\}d\mu_n\\
&\geq\liminf_{n \to \infty}\frac{1}{m}\int \log \|\mathcal{L}_{\mathsf{A}_n}(x,m)\|\,d\mu_n\\
&\geq \liminf_{n \to \infty} \log\varrho(\mathsf{A}_n) = \log\varrho(\mathsf{A}),\end{align*}
since by hypothesis we have $\mu_n \in \mathcal{M}_{\max}(\mathsf{A}_n)$. Applying the monotone convergence theorem we find that for each $m \geq 1$
\[\frac{1}{m}\int \log\|\mathcal{L}_{\mathsf{A}}(x,m)\|\,d\mu = \lim_{k \to \infty} \frac{1}{m}\int \max\{\log\|\mathcal{L}_{\mathsf{A}}(x,m)\|,-k\}d\mu \geq \log\varrho(\mathsf{A}),\]
and since $m$ is arbitrary we conclude that $\mu \in \mathcal{M}_{\max}(\mathsf{A})$. The proof is complete.
\end{proof}

\begin{proof}[ of Theorem \ref{1rc}] We shall show that if $(\lambda_n)$ is a sequence of elements of $\Delta$ converging to a limit $\lambda \in \Delta$, then $\mathfrak{r}(\lambda_n)$ converges to $\mathfrak{r}(\lambda)$. Given such a sequence, for each $n \geq 1$ let us choose a measure $\mu_n \in \mathcal{M}_{\max}(\mathsf{A}_{\lambda_n})$. Let $(n_j)$ be a subsequence of the natural numbers such that $\mathfrak{r}(\lambda_{n_j})$ converges as $j \to \infty$ to a limit $\gamma \in [0,1]$. Replacing $(n_j)$ with a finer subsequence if necessary, and making use of the fact that $\mathcal{M}_\sigma$ is compact and metrisable, we may assume that $(\mu_{n_j})$ converges to a limit $\mu \in \mathcal{M}_\sigma$. By hypothesis we have $ \lim_{j \to \infty} \max\{\|A_k(\lambda)-A_k(\lambda_{n_j})\|\colon 1 \leq k \leq \ell\}=0$, so we may apply Lemma \ref{technobanana} to deduce that $\mu \in \mathcal{M}_{\max}(\mathsf{A}_\lambda)$. Since $\chi_{[i]} \colon \Sigma_\ell \to \mathbb{R}$ is continuous and $\mathfrak{r}(\lambda)$ is the $i$-ratio of $\mathsf{A}_\lambda$, by weak-* convergence we have $\mathfrak{r}(\lambda)=\mu([i])=\lim_{j \to \infty} \int \chi_{[i]}\,d\mu_{n_j} = \lim_{j \to \infty} \mathfrak{r}(\lambda_{n_j})=\gamma$. We conclude that $\mathfrak{r}(\lambda)$ is the only limit point of the bounded sequence $(\mathfrak{r}(\lambda_n))$, and therefore that sequence converges to $\mathfrak{r}(\lambda)$. The proof is complete.
\end{proof}

\section{Applications to the uniqueness of Barabanov norms}\label{appbara}

In this section we shall say that $\mathsf{A} \subset \Mat_d(\mathbb{C})$ is \emph{reducible} if there exists a nontrivial proper subspace of $\mathbb{C}^d$ which is preserved by every element of $\mathsf{A}$; otherwise we shall say that $\mathsf{A}$ is \emph{irreducible}. If $\mathsf{A} \subset \Mat_d(\mathbb{C})$ is a nonempty bounded set, we say that a norm $\vvv\cdot\vvv$ on $\mathbb{C}^d$ is a \emph{Barabanov norm} if the property $\varrho(\mathsf{A})\vvv v \vvv = \max_{A \in \mathsf{A}}\vvv Av \vvv$ is satisfied for all $v \in \mathbb{C}^d$. A fundamental result due to N. E. Barabanov demonstrates that if $\mathsf{A} \subset \Mat_d(\mathbb{C})$ is compact and irreducible, then it admits at least one Barabanov norm \cite{Ba}. The properties and applications of Barabanov norms were subsequently investigated in depth by F. Wirth and E. Plischke \cite{Wirth4,Wirth1,Wirth2,Wirth3}, and algorithms for their computation have been suggested by V. S. Kozyakin \cite{Koz2,Koz6}.

Since clearly any positive scalar multiple of a Barabanov norm is also a Barabanov norm, we shall say that $\mathsf{A}$ has a `unique' Barabanov norm to mean that all Barabanov norms for $\mathsf{A}$ are directly proportional to one another.
In \cite{MBara}, the present author investigated the question of when a finite irreducible set of matrices $\mathsf{A}$ admits a unique Barabanov norm.
 In this section we show that a sufficient condition given in that article for a finite set $\mathsf{A}$ to have a unique Barabanov norm can be described in terms of the associated Mather set $Z_{\mathsf{A}}$. This allows us to prove uniqueness of the Barabanov norm for a one-parameter family of pairs of matrices which was studied intensively in the articles \cite{BTV,HMST,Tthesis}.

Let $\mathsf{A}=\{A_1,\ldots,A_m\} \subset \Mat_d(\mathbb{C})$. In \cite{MBara}, we said that $\mathsf{A}$ has the \emph{unbounded agreements property} if the following property holds: if $j_1,j_2 \colon \mathbb{N} \to \{1,\ldots,m\}$ are sequences such that $\limsup_{n\to\infty} \|A_{j_i(n)}\cdots A_{j_i(1)}\|\varrho(\mathsf{A})^{-n} > 0$, then for each $k \geq 1$ there exist $\ell_1,\ell_2 \geq 1$ such that $j_1(\ell_1+i)=j_2(\ell_2+i)$ for all $1 \leq i \leq k$. This definition is motivated by the following result:
\begin{theorem}[(\cite{MBara})]\label{UBN}
Let $\mathsf{A} = \{A_1,\ldots,A_m\} \subset \Mat_d(\mathbb{C})$. Suppose that $\mathsf{A}$ is irreducible and has the unbounded agreements property, and the set of matrices $\wedge^2 \mathsf{A}:=\{\wedge^2 A_1,\ldots,\wedge^2 A_m\}$ satisfies $\varrho(\wedge^2\mathsf{A}) < \varrho(\mathsf{A})^2$. Then $\mathsf{A}$ has a unique Barabanov norm.
\end{theorem}

The unbounded agreements property admits the following description in terms of the Mather set $Z_{\mathsf{A}}$:
\begin{proposition}\label{UAP}
Let $\mathsf{A}= \{A_1,\ldots,A_m\} \subset \Mat_d(\mathbb{C})$. Then $\mathsf{A}$ has the unbounded agreements property if and only if $\mathsf{Z}_{\mathsf{A}} \subseteq \Sigma_m$ contains a unique minimal set.
\end{proposition}
\begin{proof}
Let us suppose that $Z_{\mathsf{A}}$ contains a unique minimal set, and choose sequences $x^{(1)}=(x_j^{(1)})$ and $x^{(2)}=(x_j^{(2)}) \in \Sigma_m$ such that $\limsup_{n\to\infty} \left\|\mathcal{L}_{\mathsf{A}}\left(x^{(i)},n\right)\right\|\varrho(\mathsf{A})^{-n}>0$ for $i=1,2$. Let us write $Y_i:=\overline{\{\sigma^\ell x^{(i)} \colon \ell \geq 0\}}$ for each $i$. Clearly,
\[\limsup_{n\to \infty} \sup_{y \in Y_i} \frac{1}{n}\log\|\mathcal{L}_{\mathsf{A}}(y,n)\| \geq \limsup_{n \to\infty}\frac{1}{n}\log\left\|\mathcal{L}_{\mathsf{A}}\left(x^{(i)},n\right)\right\| = \log\varrho(\mathsf{A})\]
for $i=1,2$, and consequently
\[\lim_{n \to \infty} \sup_{y \in Y_i} \frac{1}{n}\log\|\mathcal{L}_{\mathsf{A}}(y,n)\| = \log\varrho(\mathsf{A})\]
since this limit exists by Lemma \ref{Fuckite}, and its value can be at most $\log\varrho(\mathsf{A})$ by the definition of $\varrho(\mathsf{A})$. Applying Theorem \ref{StSt} with $X:=Y_i$ and $f_n(x):=\log\|\mathcal{L}_{\mathsf{A}}(x,n)\|$ we deduce that for each $i=1,2$ there exists a $\sigma$-invariant Borel probability measure $\mu_i$ on $Y_i$ such that $\inf_{m \geq 1}\frac{1}{m}\int\log\|\mathcal{L}_{\mathsf{A}}(x,m)\|\,d\mu_i(x) = \log\varrho(\mathsf{A})$. In particular each $\mu_i$ belongs to $\mathcal{M}_{\max}(\mathsf{A})$, and it follows that for each $i$ the set $\supp \mu_i \subseteq Y_i$ is a subset of $Z_{\mathsf{A}}$. It follows from a classic theorem of Birkhoff (\cite[p.130]{KH}) that for each $i$ the compact $\sigma$-invariant set $\supp \mu_i$ contains a minimal set. Since by hypothesis there is exactly one minimal set contained in $Z_{\mathsf{A}}$, we conclude that $\supp \mu_1 \cap \supp \mu_2$ is nonempty. Choose any point $z \in \supp \mu_1 \cap \supp \mu_2 \subseteq Y_1 \cap Y_2$, and let $k \geq 1$ be any integer. By definition every element of $Y_1$ belongs to the closure of $\left\{\sigma^\ell x^{(1)} \colon \ell \geq 0\right\}$, and therefore there exists $\ell_1 \geq 1$ such that $d\left(\sigma^{\ell_1}x^{(1)},z\right)<2^{-k}$, implying that  $x^{(1)}_{\ell_1+j}=z_j$ for $j=1,\ldots,k$. Equally there must exist $\ell_2 \geq 1$ such that $d\left(\sigma^{\ell_2}x^{(2)},z\right)<2^{-k}$, and it follows that $x^{(2)}_{\ell_2+j}=z_j=x^{(1)}_{\ell_1+j}$ for $j=1,\ldots,k$. Since $k$ is arbitrary we conclude that $\mathsf{A}$ has the unbounded agreements property as required.

Let us now prove the converse direction. By the aforementioned theorem of Birkhoff, $Z_{\mathsf{A}}$ contains at least one minimal set. Let us suppose that $Z_{\mathsf{A}}$ contains two distinct minimal sets, $Z_1$ and $Z_2$. If $Z_1 \cap Z_2$ were nonempty it would be a $\sigma$-invariant closed proper subset of either $Z_1$ or $Z_2$, contradicting minimality, and it follows that $Z_1$ and $Z_2$ must be pairwise disjoint. Choose any $x^{(1)} \in Z_1$ and $x^{(2)}\in Z_2$. By Theorem \ref{struc} each $x^{(i)}$ is strongly extremal and hence satisfies $\limsup_{n\to\infty} \left\|\mathcal{L}_{\mathsf{A}}(x^{(i)},n)\right\|\varrho(\mathsf{A})^{-n}>0$. Since $Z_1$ and $Z_2$ are compact and do not intersect, we may choose $k>0$ such that $d(y,z) >2^{-k}$ whenever $y \in Z_1$ and $z \in  Z_2$. In particular, since every point of the form $\sigma^j x^{(i)}$ belongs to $Z_i$, we have $d\left(\sigma^{\ell_1}x^{(1)},\sigma^{\ell_2}x^{(2)}\right)>2^{-k}$ for all $\ell_1,\ell_2 \geq 1$. It follows that for every $\ell_1,\ell_2 \geq 1$ we must have $x_{\ell_1+j}^{(1)} \neq x_{\ell_2+j}^{(2)}$ for some $j$ in the range $1,\ldots,k$, and as such $\mathsf{A}$ does not have the unbounded agreements property. The proof is complete.
\end{proof}
We immediately deduce the following:
\begin{corollary}
Let $\mathsf{A}= \{A_1,\ldots,A_m\}\subset \Mat_d(\mathbb{C})$ and suppose that $\mathcal{M}_{\max}(\mathsf{A})$ contains exactly one measure. Then $\mathsf{A}$ has the unbounded agreements property. 
\end{corollary}
\begin{proof}
If $\mathsf{A}$ does not have the unbounded agreements property, then by Proposition \ref{UAP} there exist two distinct minimal sets $Y_1, Y_2 \subseteq Z_{\mathsf{A}}$.  By the Krylov-Bogolioubov theorem there exist $\sigma$-invariant Borel probability measures $\mu_1$ and $\mu_2$ supported in $Y_1$ and $Y_2$ respectively. Since two distinct minimal sets cannot intersect one another we have $\supp \mu_1 \cap \supp \mu_2 = \emptyset$ and so in particular $\mu_1 \neq \mu_2$. However, by Theorem \ref{struc} we have $\mu_1,\mu_2 \in \mathcal{M}_{\max}(\mathsf{A})$ since $\mu_i(Z_\mathsf{A})=1$ for $i=1,2$, and since $\mu_1 \neq \mu_2$ this contradicts our hypothesis.
\end{proof}
In \cite{MBara} we proved uniqueness of the Barabanov norm for certain sets of matrices constructed specifically for that purpose. Using Proposition \ref{UAP} we are able to give the following more natural family of examples:
\begin{theorem}
For each $\alpha >0$ define $\mathsf{A}_\alpha=\{A_1,\alpha A_2\}$ where \[A_1:=\left(\begin{array}{cc}1&1\\0&1\end{array}\right),\qquad A_2:=\left(\begin{array}{cc}1&0\\1&1\end{array}\right).\]
Then for each $\alpha \in (0,1]$ the pair $\mathsf{A}_\alpha$ has a unique Barabanov norm.
\end{theorem}
\begin{proof}
By computing the eigenspaces of each matrix one may establish that each of the pairs $\mathsf{A}_\alpha$ is irreducible, and hence in particular has at least one Barabanov norm. It was shown in \cite[Theorem 2.3]{HMST} that for each $\alpha \in (0,1]$ there exists a unique minimal set $X_{\mathfrak{r}(\alpha)}\subset \Sigma_2$ such that if $x \in \Sigma_2$ is recurrent and strongly extremal for $\mathsf{A}_\alpha$, then $x \in X_{\mathfrak{r}(\alpha)}$. Since $Z_{\mathsf{A}_\alpha}$ is closed and $\sigma$-invariant, it contains at least one minimal set. If $Y \subseteq Z_{\mathsf{A}_\alpha}$ is minimal and $x \in Y$ then $x$ is recurrent by minimality and strongly extremal by Theorem \ref{struc}, and it follows that $x \in X_{\mathfrak{r}(\alpha)}$. In particular $Y \cap X_{\mathfrak{r}(\alpha)} \neq \emptyset$ and therefore $Y=X_{\mathfrak{r}(\alpha)}$ by minimality. We conclude that $Z_{\mathsf{A}_\alpha}$ contains a unique minimal set, and by Proposition \ref{UAP} it follows that $\mathsf{A}_\alpha$ has the unbounded agreements property. 

To prove the proposition it remains only to show that the joint spectral radius of the set $\wedge^2 \mathsf{A}_\alpha:=\{\wedge^2 A_1, \wedge^2 (\alpha A_2)\}$ is strictly less than $\varrho(\mathsf{A}_\alpha)^2$. However, since the matrices $A_1,A_2$ are two-dimensional, $\wedge^2 \mathsf{A}_\alpha$ simply consists of the two one-dimensional matrices with entries respectively equal to $1=\det A_1$ and $\alpha = \det (\alpha A_2)$. It follows that $\varrho(\wedge^2 \mathsf{A}_\alpha)=1$. On the other hand by \cite[Lemma 3.4]{HMST} we have $\varrho(\mathsf{A}_\alpha)>1$ for each $\alpha \in (0,1]$. We conclude that the conditions of Theorem \ref{UBN} are met, and for each $\alpha \in (0,1]$ the pair $\mathsf{A}_\alpha$ admits a unique Barabanov norm.
\end{proof}
\emph{Remark}. It was shown in \cite{BTV,HMST} that for certain choices of $\alpha \in (0,1]$ the pair $\mathsf{A}_\alpha$ does not have the finiteness property. The above proposition therefore implies the existence of a pair of matrices which admits a unique Barabanov norm but does not have the finiteness property. This result was previously claimed without proof in the article \cite{MBara}.

\appendix
\section{Subadditive ergodic optimisation}
Throughout this section, we let $T \colon X \to X$ be a continuous transformation of a compact metric space. We denote by $\mathcal{M}$ the set of all Borel probability measures on $X$. We equip this set with the weak-* topology, which is the smallest topology such that $\mu \mapsto \int g\,d\mu$ is a continuous map from $\mathcal{M}$ to $\mathbb{R}$ for every $g \in C(X)$. Under the weak-* topology $\mathcal{M}$ is a compact metrisable space (see e.g. \cite{Bil,Parth}). We let $\mathcal{M}_T\subseteq \mathcal{M}$ denote the set of all measures which are invariant with respect to $T$, and we define $\mathcal{E}_T\subseteq \mathcal{M}_T$ to be the set of all invariant measures with respect to which $T$ is ergodic. 

\emph{Ergodic optimisation} is concerned with the following general problem: given a continuous function $f \colon X \to \mathbb{R}$, what can we say about the maximum ergodic average $\beta(f):=\max_{\mu \in \mathcal{M}_T} \int f\,d\mu$, the set of measures which attain this average, and the set of points $x \in X$ for which $\lim_{n\to \infty} \frac{1}{n}\sum_{i=0}^{n-1} f(T^ix) = \beta(f)$? An overview of this research area may be found in \cite{J}; some articles in this area of particular note include \cite{B1,BJX,CLT,YH}. In this appendix we are concerned with the more general problem of identifying points $x \in X$ and measures which maximise the growth of \emph{subadditive} ergodic averages, which we describe in detail below. The theorems proved in this section mainly consist of extensions of results already present in the literature, and as such are not of the first order of originality; on the other hand, these results are sufficiently far removed from their antecedents in the literature that we feel it would be unreasonable to ask the reader to accept them without proof. For this reason we have decided to present them as an appendix separate from the main body of the article.  

Recall that a sequence $(a_n)$ of elements of $\mathbb{R}\cup\{-\infty\}$ is called \emph{subadditive} if one has $a_{n+m} \leq a_n + a_m$ for all $n,m \geq 1$. The following classical result due to Fekete \cite{Fekete} is used numerous times throughout this article. We include the proof in view of its brevity.
\begin{lemma}\label{Fuckite}
Let $(a_n)$ be a subadditive sequence of elements of $\mathbb{R}\cup\{-\infty\}$. Then $\lim_{n \to \infty} a_n/n = \inf_{n \geq 1} a_n/n \in \mathbb{R} \cup \{-\infty\}$.
\end{lemma}
\begin{proof}
Given a real number $\lambda>\inf_{n\geq 1} a_n/n$, choose an integer $m$ such that $a_m \leq m\lambda$. For each $n \geq 1$ write $n=q_nm+r_n$ where $q_n:=\lfloor n/m\rfloor$ and $0 \leq r_n <m$. We have $a_n \leq q_na_m + a_{r_n} \leq n\lambda + r_na_1$ for all $n \geq 1$, and it follows that $\limsup_{n \to \infty} a_n/n \leq \lambda$. Since $\lambda$ is arbitrary we conclude that $\limsup_{n \to \infty} a_n/n \leq \inf_{n \geq 1}a_n/n$, and the result follows.
\end{proof} 
A sequence $(f_n)$ of functions from $X$ to $\mathbb{R}\cup\{-\infty\}$ will be called \emph{subadditive} if for each $x \in X$ and $n, m \geq 1$ the inequality $f_{n+m}(x) \leq f_n(T^mx)+f_m(x)$ is satisfied. A starting point for our study of subadditive sequences of functions is the following classical result: 
\begin{theorem}[(Subadditive ergodic theorem)]\label{SAET}
Let $(X,\mathcal{F},\mu)$ be a probability space equipped with an ergodic measure-preserving transformation $T$. Suppose that $(f_n)$ is a sequence of measurable functions from $X$ to $\mathbb{R}\cup\{-\infty\}$ such that $\max\{f_n,0\}$ is integrable for each $n \geq 1$, and such that for each $n, m \geq 1$ one has $f_{n+m}(x) \leq f_n(T^mx)+f_m(x)$ for $\mu$-a.e. $x$. Then $\lim_{n \to \infty}f_n(x)/n = \inf_{n \geq 1}(1/n)\int f_n\,d\mu$ for $\mu$-a.e. $x$.
\end{theorem}
\emph{Remark}. It is more usual to state this theorem under the additional hypotheses that each $f_n$ is integrable, and $\inf_{n \geq 1}(1/n)\int f_n\,d\mu > -\infty$. Some proofs of Theorem \ref{SAET} in this form may be found in e.g. \cite{Der,KW}. To deduce the version of Theorem \ref{SAET} given above, it suffices to show that if the additional hypotheses do not hold then necessarily $\limsup_{n \to \infty}f_n(x)/n = -\infty$ a.e, which may be achieved by taking upper estimates using the Birkhoff ergodic theorem.

Theorem \ref{SAET} indicates that given a subadditive sequence of upper semi-continuous functions $(f_n)$ and an invariant measure $\mu \in \mathcal{M}_T$, the quantity $\inf_{n \geq 1}\frac{1}{n}\int f_n\,d\mu$ may be understood as the ergodic average of $(f_n)$ with respect to $\mu$. This motivates the following definition: given a transformation $T\colon X \to X$ and a subadditive sequence of upper semi-continuous functions $(f_n)$ from $X$ to $\mathbb{R} \cup \{-\infty\}$, we define the \emph{maximum ergodic average} of the sequence $(f_n)$ to be the quantity
\[\beta[(f_n)]:= \sup_{\mu \in \mathcal{M}_T}\inf_{n \geq 1} \frac{1}{n}\int f_n\,d\mu.\]
We shall say that $\mu \in \mathcal{M}_T$ is a \emph{maximising measure} for $(f_n)$ if $\inf_{n \geq 1}\frac{1}{n}\int f_n\,d\mu = \beta[(f_n)]$, and denote the set of all maximising measures by $\mathcal{M}_{\max}[(f_n)]$.
The key objective of this section is to prove the following characterisation of $\beta[(f_n)]$:
\begin{theorem}\label{StSt}
Let $(f_n)$ be a subadditive sequence of upper semi-continuous functions taking values in $\mathbb{R} \cup \{-\infty\}$. Then,
\begin{align*}
\beta[(f_n)]&=\sup_{\mu \in \mathcal{E}_T} \inf_{n \geq 1}\frac{1}{n}\int f_n\,d\mu = \inf_{n \geq 1}\sup_{x \in X}\frac{1}{n}f_n(x) \\
&= \inf_{n \geq 1}\sup_{\mu \in \mathcal{M}_T}\frac{1}{n}\int f_n\,d\mu = \sup_{x \in X} \inf_{n \geq 1} \frac{1}{n}f_n(x).\end{align*}
In all but the last of these expressions, the infimum over all $n\geq 1$ may be replaced with the limit as $n \to \infty$ of the same quantity, without altering the value of the expression. Furthermore, every supremum arising in each of the above expressions is attained.
\end{theorem}
Theorem \ref{StSt} may be seen as a subadditive analogue of \cite[Proposition 2.1]{J}. Theorem \ref{StSt} extends previous results in the subadditive context which deal with the case in which each $f_n$ is continuous and takes values only in $\mathbb{R}$. Under these additional hypotheses the identity
\begin{equation}\label{bollow}\beta[(f_n)]=\sup_{\mu \in \mathcal{E}_T} \inf_{n \geq 1}\frac{1}{n}\int f_n\,d\mu= \inf_{n \geq 1}\sup_{x \in X}\frac{1}{n}f_n(x)=\sup_{x \in X} \limsup_{n \to \infty} \frac{1}{n}f_n(x)\end{equation}
was previously established by S. J. Schreiber \cite{Sch}. The same relationships were later independently proved by R. Sturman and J. Stark under the same hypotheses \cite{StSt}. Our removal of the condition that each $f_n$ takes only real values has the significant advantage that we may treat linear cocycles of matrices or operators which may not be invertible or even nonzero at every point. We also remove the assumption that each $f_n$ is lower semi-continuous. The attainment of the suprema in \eqref{bollow}, which was not addressed in the work of Schreiber and Sturman-Stark, is also crucial in several of our applications. The identities
\[\beta[(f_n)]=\sup_{x \in X} \inf_{n \geq 1} \frac{1}{n}f_n(x)=\inf_{n \geq 1}\sup_{x \in X}\frac{1}{n}f_n(x)\]
and the attainment of the corresponding suprema are therefore original with this document.

In proving Theorem \ref{StSt} our approach essentially follows the method of Schreiber, as opposed to the somewhat different path taken by Sturman and Stark. This is motivated by the fact that this proof takes us naturally though a result - Lemma \ref{wkx} below - which has independent interest, and is applied separately from Theorem \ref{StSt} in the main document. However, our proof differs from those given by Schreiber and Sturman-Stark in that we do not actually make use of the subadditive ergodic theorem in proving Theorem \ref{StSt}.

Before proving the identities which form the main part of Theorem \ref{StSt}, we shall prove some basic properties of the set of maximising measures $\mathcal{M}_{\max}[(f_n)]$. We begin with the following simple lemma:
\begin{lemma}\label{USC}
If $g \colon X \to \mathbb{R} \cup \{-\infty\}$ is upper semi-continuous, then the map from $\mathcal{M}$ to $\mathbb{R}\cup\{-\infty\}$ given by $\mu \mapsto\int g\,d\mu$ is upper semi-continuous. If $(f_n)$ is a subadditive sequence of upper semi-continuous functions from $X$ to $\mathbb{R} \cup \{-\infty\}$, then the map from $\mathcal{M} \to \mathbb{R} \cup \{-\infty\}$ given by $\mu \mapsto \inf_{m \geq 1}\frac{1}{m}\int f_m\,d\mu$ is also upper semi-continuous.
\end{lemma}
\begin{proof}
A function from a metrisable space to $\mathbb{R} \cup \{-\infty\}$ is upper semi-continuous if and only if it is equal to the pointwise limit of a decreasing sequence of continuous functions taking values in $\mathbb{R}$ (see e.g. \cite[ch. IX]{Bourbaki}). Given a function $g$ as above, let $(g_i)_{i=1}^\infty$ be such a decreasing sequence converging pointwise to $g$. For each $i$ the map $\mu \mapsto \int g_i\,d\mu$ is clearly real-valued, and is by definition weak-* continuous. For each $\mu \in \mathcal{M}$ the real-valued sequence $(\int g_i\,d\mu)_{i=1}^\infty$ decreases to $\int g\,d\mu$ by the monotone convergence theorem, and it follows that the map $\mu \mapsto \int g\,d\mu$ is upper semi-continuous. If $(f_n)$ is a subadditive sequence as above then each of the functions $\mu \mapsto \frac{1}{m}\int f_m\,d\mu$ is upper semi-continuous by the preceding argument, and since the pointwise infimum of a family of upper semi-continuous functions is also upper semi-continuous the result follows.\end{proof}
We immediately deduce:
\begin{proposition}\label{mmax}
Let $(f_n)$ be a subadditive sequence of upper semi-continuous functions taking values in $\mathbb{R} \cup \{-\infty\}$. Then $\mathcal{M}_{\max}[(f_n)]$ is compact, convex and nonempty, and the extreme points of $\mathcal{M}_{\max}[(f_n)]$ are precisely the ergodic elements of $\mathcal{M}_{\max}[(f_n)]$.\end{proposition}
\begin{proof}
It follows from the semicontinuity of the map $\mu \mapsto \inf_{m \geq 1}\frac{1}{m}\int f_m\,d\mu$ that $\mathcal{M}_{\max}(f_m)$ is compact and nonempty. By subadditivity we have $\mu \in \mathcal{M}_{\max}[(f_m)]$ if and only if $\lim_{n \to \infty} \frac{1}{n}\int f_n\,d\mu = \beta[(f_n)]$, and it is clear from this that if $m_1,m_2 \in \mathcal{M}_{\max}[(f_m)]$ then $\lambda m_1 + (1-\lambda)m_2 \in \mathcal{M}_{\max}[(f_m)]$ for every $\lambda \in [0,1]$. 

If $\mu \in \mathcal{M}_{\max}[(f_n)]$ is ergodic then it is well-known that $\mu$ is an extreme point of $\mathcal{M}_T$ (see e.g. \cite[p.152]{W}), and hence is also an extreme point of $\mathcal{M}_{\max}[(f_n)]$. Conversely, if $\mu$ is a non-ergodic maximising measure, let $A \subseteq X$ be a Borel set such that $T^{-1}A=A$ up to $\mu$-measure zero and $0<\mu(A)<1$. Define measures $m_1,m_2 \in \mathcal{M}_T$ by $m_1(Y):=\mu(A \cap Y) / \mu(A)$ and $m_2(Y):= \mu(Y \setminus A) / \mu(X \setminus A)$. An easy argument by contradiction shows that both $m_1$ and $m_2$ must belong to $\mathcal{M}_{\max}[(f_m)]$, and since $\mu = \mu(A)m_1 + (1-\mu(A))m_2$ it follows that $\mu$ is not an extreme point of $\mathcal{M}_{\max}[(f_n)]$.
\end{proof}
The next two results together constitute the bulk of the proof of Theorem \ref{StSt}, but each has additional usefulness in its own right. Lemma \ref{wkx} below is a modification of a result of Schreiber \cite[Lemma 1]{Sch}, although some differences exist since parts of Schreiber's proof require the quantity $\inf f_n$ to be finite for each $n$.
\begin{lemma}\label{wkx}
Let $(f_n)$ be a subadditive sequence of upper semi-continuous functions taking values in $\mathbb{R} \cup \{-\infty\}$, let $x \in X$, and suppose that $\lim_{n \to \infty}(1/n)f_n(x) = c$. Define a sequence of measures $\mu_n \in \mathcal{M}$ by $\mu_n:=(1/n)\sum_{k=0}^{n-1}\delta_{T^kx}$, where $\delta_z$ denotes the Dirac measure concentrated at the point $z \in X$, and suppose that $\mu \in \mathcal{M}$ is a weak-* accumulation point of $(\mu_n)$. Then $\mu \in \mathcal{M}_T$ and $\inf_{m \geq 1}\frac{1}{m}\int f_m\,d\mu \geq c$.
\end{lemma}
\begin{proof}
The proof that every accumulation point of $(\mu_n)$ must lie in $\mathcal{M}_T$ is entirely standard, see for example \cite[p.151]{W}. If $c=-\infty$ then this suffices to complete the proof, so for the remainder of the proof we assume that $c$ is real. To prove that $\inf_{m \geq 1}\frac{1}{m}\int f_m\,d\mu \geq c$ for every accumulation point $\mu$, it suffices to prove that $\liminf_{n \to \infty}\frac{1}{m}\int f_m\,d\mu_n \geq c$ for every $m \geq 1$ and apply the first clause of Lemma \ref{USC}.

We begin with the following claim: if $k,\ell \geq 1$ then $\lim_{n \to \infty}\frac{1}{n}f_k(T^{n-\ell}x)=0$. To see this, note that for large enough $n$ subadditivity yields $f_{n+k-\ell}(x) \leq f_{n-\ell}(x)+f_k(T^{n-\ell}x) \leq f_{n-\ell}(x)+\sup f_k$, and since clearly $\lim_{n \to \infty}\frac{1}{n}f_{n+k-\ell}(x)=\lim_{n \to \infty}\frac{1}{n}(f_{n-\ell}(x)+\sup f_k)=c$ we obtain the desired result.

Let us now fix $m \geq 1$, and show that $\liminf_{n \to \infty}\frac{1}{nm}\sum_{i=0}^{n-1}f_m(T^ix) \geq c$. Since
\begin{align*}\lim_{n\to \infty}\frac{1}{nm}\left(\sum_{i=0}^{n-1}f_m(T^ix) - \sum_{i=0}^{n-m+1}f_m\left(T^ix\right)\right) &= \lim_{n \to \infty}\frac{1}{nm}\sum_{i=n-m+2}^{n-1}f_m\left(T^ix\right)\\&=\lim_{n \to \infty}\frac{1}{nm}\sum_{i=1}^{m-2}f_m\left(T^{n-i}x\right)=0\end{align*}
by our previous claim, it suffices to show that $\liminf_{n \to \infty}\frac{1}{nm}\sum_{i=0}^{n-m+1}f_m(T^ix) \geq c$. For each $i$ in the range $0 \leq i <m$, then, let us choose integers $q_i$, $r_i$ such that $n=i+q_im+r_i$ with $q_i \geq 0$ and $0 \leq r_i <m$. We have
\[\sum_{i=0}^{m-1}\sum_{j=0}^{q_i-1}f_m\left(T^{i+jm}x\right) = \sum_{i=0}^{n-m+1}f_m(T^ix),\]
and so using subadditivity we may estimate
\begin{align*}m f_n(x) &\leq \sum_{i=0}^{m-1}\left(f_i(x) +\sum_{j=0}^{q_i-1}f_{m}\left(T^{i+jm}x\right) + f_{r_i}\left(T^{i+q_im}x\right)\right) \\
&= \sum_{i=0}^{m-1}f_i(x) +\sum_{i=0}^{n-m+1}f_{m}\left(T^ix\right) + \sum_{i=0}^{m-1}f_{r_i}\left(T^{n-r_i}x\right).\end{align*}
Dividing both sides by $nm$, taking the limit inferior as $n \to \infty$ and applying our previous claim finishes the proof.
\end{proof}

 \begin{proposition}\label{heavy}
Let $(f_n)$ be a subadditive sequence of upper semi-continuous functions taking values in $\mathbb{R}\cup \{-\infty\}$, and define $\tilde \beta[(f_n)]:=\lim_{n \to \infty} \sup_{x \in X}\frac{1}{n}f_n(x)$. Then there exists $z \in X$ such that $\inf_{m \geq 1}\frac{1}{m}f_m(z) \geq \tilde\beta[(f_n)]$.
\end{proposition}
\begin{proof}
To simplify some expressions in this proof we shall denote the value $\tilde \beta[(f_n)]$ simply by $\tilde\beta$. If $\tilde \beta = -\infty$ then it suffices to choose any $z \in X$, so we shall assume that this value is finite. 
To prove the Proposition we shall assume that no point $z \in X$ exists with the required property, and show that this leads to a contradiction. 

Making this hypothesis, it follows that for every $x \in X$ there exists an integer $n(x) \geq 1$ and real number $\varepsilon(x)>0$ such that $f_{n(x)}(x) < n(x)(\tilde \beta - \varepsilon(x))$. Since each $f_k$ is upper semi-continuous, it follows that for each $x \in X$ there exists an open subset $U_x$ of $X$ containing $x$ such that $f_{n(x)}(y) < n(x)(\tilde\beta-\varepsilon(x))$ for all $y \in U_x$. By compactness we may choose finitely many points $x_1,\ldots,x_k \in X$ such that $\bigcup_{i=1}^k U_{x_i}=X$. It follows that we may choose a real number $\varepsilon>0$ and a function $r_1 \colon X \to \mathbb{N}$ taking values in the set $\{n(x_1),n(x_2),\ldots,n(x_k)\}$ such that $f_{r_1(x)}(x) <r_1(x)(\tilde \beta - \varepsilon)$ for every $x \in X$. Let $m, M \geq 1$ be natural numbers such that $m \leq r_1(x) < M$ for every $x$.

Using the function $r_1$, we shall now define a sequence of functions $r_k \colon X\to \mathbb{N}$ for all $k \geq 2$. The function $r_k$ having been defined, let us define $r_{k+1}$ by setting $r_{k+1}(x):=r_k(x)+r_1\left(T^{r_k(x)}x\right)$ for all $x \in X$. An elementary induction shows that $km \leq r_k(x) \leq k(M-1)$ for every $x \in X$ and $k \geq 1$. We claim that additionally $f_{r_k(x)}(x)<r_k(x)\left(\tilde\beta - \varepsilon\right)$ for all $x \in X$ and $k \geq 1$. In the case $k=1$ this has already been established. Given that this relation holds for all $x \in X$ for some fixed $k \geq 1$, note that for each $x \in X$ we have
\begin{align*}f_{r_{k+1}(x)}(x) &\leq f_{r_1\left(T^{r_k(x)}x\right)}\left(T^{r_k(x)}x\right) + f_{r_k(x)}(x)\\&\leq r_1\left(T^{r_k(x)}x\right)\left(\tilde \beta - \varepsilon\right) + r_k(x)\left(\tilde \beta - \varepsilon\right) = r_{k+1}(x)\left(\tilde \beta - \varepsilon\right),\end{align*}
which establishes the required inequality for all $x \in X$ in the case $k+1$. The result follows for all $x \in X$ and $k \geq 1$ by induction on $k$.

Now, for each $k \geq 1$ define $\tilde\beta_k:=\sup_{i \geq k} \sup_{x \in X}\frac{1}{i}f_i(x)$ and note that the sequence $(\tilde\beta_k)$ decreases to the limit $\tilde\beta$ as $k \to \infty$. Define also $\ell_k(x):=kM-r_k(x)$ for every $k \geq 1$ and $x \in X$, and note that $k \leq \ell_k(x) \leq k(M-m)$ as a consequence of the bounds on $r_k(x)$. For all $x \in X$ and $k \geq 1$ we have the inequality
\begin{align*}f_{kM}(x) &\leq f_{r_k(x)}(x) + f_{\ell_k(x)}\left(T^{r_k(x)}x\right) \leq r_k(x)\left(\tilde \beta - \varepsilon\right) + \ell_k(x)\tilde\beta_{k}\\
&=kM\tilde\beta -r_k(x)\varepsilon + \ell_k(x)(\tilde\beta_k - \tilde \beta) \leq kM\tilde\beta - km\varepsilon + k(M-m)(\tilde\beta_k-\tilde\beta).\end{align*}
Taking the supremum over $x \in X$, dividing both sides by $kM$ and taking the limit as $k \to \infty$ we find that $\sup_{x \in X}\frac{1}{kM}f_{kM}(x)<\tilde\beta$ for all large enough $k$, which is a contradiction. The proof is complete.
\end{proof}
\emph{Remark}. Proposition \ref{heavy} generalises a lemma of Y. Peres; previous proofs of that lemma relied on the maximal ergodic theorem \cite{Peres,Ralston1}. Note that the point $z$ guaranteed by Proposition \ref{heavy} may be unique: if $T \colon X \to X$ is minimal, then for the sequence $(f_n)$ defined by $f_n(x):=d(T^nx,z_0)-d(x,z_0)$, the only point $z \in X$ such that $f_m(z) \geq m\tilde\beta[(f_n)]=0$ for all $m \geq 1$ is $z_0$. 

By modifying the standard proof of the existence of minimal subsets of dynamical systems, we can obtain the following interesting extension of Proposition \ref{heavy}:
\begin{corollary}
Let $(f_n)$ be as in Proposition \ref{heavy}. Then without loss of generality the point $z \in X$ described by Proposition \ref{heavy} may be taken to be recurrent.
\end{corollary}
\begin{proof}
Let us denote by $\mathfrak{H}$ the set of all nonempty compact sets $Z \subseteq X$ such that both $TZ \subseteq Z$ and $\inf_{n \geq 1}\sup_{x \in Z}\frac{1}{n}f_n(x)=\tilde\beta[(f_n)]$, and equip $Z$ with the partial order given by set inclusion. We shall show using Zorn's lemma that $\mathfrak{H}$ has a minimal element with respect to this ordering.

To this end, let us suppose that $\mathfrak{O}$ is a subset of $\mathfrak{H}$ in which set inclusion is a total order, and define $Y:=\bigcap_{Z \in \mathfrak{O}} Z$. Clearly $Y$ is compact and $T$-invariant, and it is nonempty since otherwise $\bigcup_{Z \in \mathfrak{O}}(X \setminus Z)$ would be an open cover of $X$ without a finite subcover. For each $m \geq 1$ the set $\{x \in X \colon f_m(x) \geq m\tilde\beta[(f_n)]\}$ is compact and
intersects every $Z \in \mathfrak{O}$, so in particular each such set must intersect $Y$; it follows from this that $\inf_{m \geq 1}\sup_{x \in Y}\frac{1}{m}f_m(x) =\tilde\beta[(f_n)]$. We conclude that $Y \in \mathfrak{H}$ and $Y \subseteq Z$ for every $Z \in \mathfrak{O}$, and it follows by Zorn's lemma that $\mathfrak{H}$ must have a minimal element with respect to inclusion.

Let $Z \in \mathfrak{H}$ be such a minimal element. By applying Proposition \ref{heavy} to $Z$, there exists $z \in Z$ such that $\inf_{m\geq 1}\frac{1}{m} f_m(z) \geq \tilde\beta[(f_n)]$. We claim that $z$ is recurrent. Let us define $\tilde Z:= \overline{\{T^k z \colon k \geq 1\}}$. Clearly $\tilde Z$ is closed and $T$-invariant. For each $m \geq 1$ we have $(m+1)\tilde\beta[(f_n)] \leq f_{m+1}(z) \leq f_1(z) + f_m(Tz)$, and hence
\[\beta[(f_n)]\geq \inf_{m \geq 1}\sup_{x \in \tilde Z} \frac{1}{m}f_m(x) = \lim_{m \to \infty} \sup_{x \in \tilde Z}\frac{1}{m}f_m(x) \geq \liminf_{m \to \infty} \frac{1}{m} f_m(Tz) \geq \tilde\beta[(f_n)],\]
from which it follows that $\tilde Z \in \mathfrak{H}$. Since by assumption $ Z$ is a minimal element of $\mathfrak{H}$ with respect to inclusion, and $\tilde Z \subseteq Z$, we conclude that necessarily $\tilde Z = Z$, and this implies that $z$ is recurrent as claimed.
\end{proof}

\begin{proof}[of Theorem \ref{StSt}]
Combining respectively Proposition \ref{heavy}, Lemma \ref{wkx} and Proposition \ref{mmax} we immediately obtain
\begin{equation}\label{durr1}\inf_{n \geq 1}\sup_{x \in X}\frac{1}{n}f_n(x) \leq \sup_{x \in X}\inf_{n \geq 1}\frac{1}{n}f_n(x) \leq \sup_{\mu \in \mathcal{M}_T} \inf_{n \geq 1}\frac{1}{n}\int f_n\,d\mu = \sup_{\mu \in \mathcal{E}_T} \inf_{n \geq 1}\frac{1}{n}\int f_n\,d\mu,\end{equation}
and the suprema in all of these expressions are attained. Now, if $n \geq 1$ and $\nu \in \mathcal{M}_T$, the inequality $\frac{1}{n}\int f_n\,d\nu \leq \sup_{\mu \in \mathcal{M}_T}\frac{1}{n}\int f_n\,d\mu \leq \sup_{x \in X}\frac{1}{n}f_n$ is obvious. Taking first the infimum over $n\geq 1$ and then the supremum over $\nu$ we obtain
\begin{equation}\label{durr2}\sup_{\nu \in \mathcal{M}_T}\inf_{n \geq 1}\frac{1}{n}\int f_n\,d\nu \leq \inf_{n \geq 1}\sup_{\mu \in \mathcal{M}_T}\frac{1}{n}\int f_n\,d\mu \leq \inf_{n \geq 1}\sup_{x \in X} \frac{1}{n}f_n(x),\end{equation}
and it follows from Lemma \ref{USC} and the compactness of $\mathcal{M}_T$ that the supremum in the middle expression is also attained for every $n$. Combining \eqref{durr1} and \eqref{durr2} proves all of the desired identities. Lastly, we note that with the single exception of the quantity $\inf_{n \geq 1}\frac{1}{n}f_n(x)$, every infimum arising in these expressions can be expressed in the form $\inf_{n\geq 1}\frac{1}{n}a_n$ where $(a_n)$ is a subadditive sequence, and hence by Lemma \ref{Fuckite} is also a limit. The proof is complete.\end{proof}

In ergodic optimisation, a continuous function $f \colon X \to \mathbb{R}$ is said to satisfy the \emph{subordination principle}, introduced in \cite{B1}, if the set $\mathcal{M}_{\max}(f):=\{\mu \in \mathcal{M}_T \colon \int f\,d\mu = \beta(f)\}$ satisfies the following property: if $\mu \in \mathcal{M}_{\max}(f)$, $\nu \in \mathcal{M}_T$ and $\supp \nu \subseteq \supp \mu$, then $\nu \in \mathcal{M}_{\max}(f)$. This is equivalent to the existence of a compact $T$-invariant set $Y \subseteq X$ such that if $\mu \in \mathcal{M}_T$ and $\supp \mu \subseteq Y$, then $\mu \in \mathcal{M}_{\max}(f)$. (For a proof of this result see \cite{M1}). The following result, which generalises an argument given in \cite{Madv}, describes an analogous phenomenon in the subadditive context.
\begin{lemma}\label{subordprin}
Let $(f_n)$ be a subadditive sequence of upper semi-continuous functions taking values in $\mathbb{R} \cup \{-\infty\}$, and suppose that there exists $\lambda \in \mathbb{R}$ such that $\sup\{f_n(x) \colon x \in X\} =n\lambda$ for infinitely many $n \geq 1$. Then $\lambda = \beta[(f_n)]$, and the set
\[Y:=\bigcap_{n=1}^\infty \left\{x \in X \colon f_n(x)=n\lambda\right\}\]
is compact, nonempty, and satisfies $TY \subseteq Y$. Moreover, for each $\mu \in \mathcal{M}_T$ we have $\mu \in \mathcal{M}_{\max}[(f_n)]$ if and only if $\supp \mu \subseteq Y$.
\end{lemma}
\begin{proof}
For each $n \geq 1$ define $Y_n:=\{x \in X \colon f_n(x)=n\lambda\}$. Using semicontinuity and the fact that $\sup f_n =n\lambda $ for each $n$ it follows that each $Y_n$ is nonempty and closed, hence compact. We claim that $Y_{n+1} \subseteq Y_n$ and $Y_{n+1} \subseteq T^{-1}Y_n$ for each $n$. To see the former, note that if $x \in Y_{n+1}$ then $(n+1)\lambda =f_{n+1}(x) \leq f_n(x) + f_1(T^nx) \leq f_n(x) + \lambda \leq (n+1)\lambda$ and therefore $f_n(x)=n\lambda$ which implies that $x \in Y_n$. To see the latter we similarly observe that if $x \in Y_{n+1}$, then since $(n+1)\lambda =f_{n+1}(x) \leq f_1(x) + f_n(Tx) \leq \lambda + f_n(Tx) \leq (n+1)\lambda$ we have $f_n(Tx)=n\lambda$ and therefore $Tx \in Y_n$. It follows that the intersection $Y:=\bigcap_{n=1}^\infty Y_n$ is compact and nonempty as claimed, and $TY \subseteq Y$. We have $\lambda=\beta[(f_n)]$ as a direct consequence of Theorem \ref{StSt}. It follows immediately that if $\mu \in \mathcal{M}_T$, then $\mu \in \mathcal{M}_{\max}[(f_n)]$ if and only if $\int f_n\,d\mu = n\lambda$ for all $n \geq 1$, if and only if $\mu(Y_n)=1$ for all $n \geq 1$, if and only if $\supp \mu \subseteq Y$. 
\end{proof}

\bibliographystyle{amsplain}
\bibliography{SAEOJSR}
\end{document}